\documentclass[a4paper,12pt]{amsart}
\usepackage{amsthm,amsmath, amssymb, amscd,mathtools, braket,mathrsfs}
\usepackage[top=30truemm, bottom=30truemm, left=25truemm, right=25truemm]{geometry}
\usepackage{color}
\usepackage{stmaryrd}
\usepackage{graphicx}
\usepackage{caption}
\usepackage{hyperref}

\newcommand{\Haus}{\dim_{\mathrm{H}}}
\newcommand{\PS}{\mathrm{PS}}

\newcommand{\diam}{\mathrm{diam}\:}
\newcommand{\calC}{\mathcal{C}}

\numberwithin{equation}{section}
\newcounter{count}

\newcommand{\num}{\stepcounter{count}\the\value{count}}

\renewcommand{\limsup}{\varlimsup}

\newtheorem{theorem}{Theorem}[section]
\newtheorem{lemma}[theorem]{Lemma}

\newtheorem{question}[theorem]{Question}

\theoremstyle{remark}
\newtheorem{remark}[theorem]{Remark}

\theoremstyle{definition}
\newtheorem{notation}[theorem]{Notation}

\begin{document}

\title[Finiteness of solutions to equations on PS sequences]{Finiteness of solutions to linear Diophantine equations on Piatetski-Shapiro sequences }

\author[K. Saito]{Kota Saito}
\email{saito.kota@nihon-u.ac.jp}

\address{Kota Saito\\ Department of Mathematics\\ College of Science $\&$ Technology \\ Nihon University\\Kanda\\ Chiyoda-ku\\ Tokyo\\
101-8308\\ Japan} 

\thanks{It is to appear in \textit{Acta Arithmetica}}

\subjclass[2020]{11D04, 11K55}
\keywords{Piatetski-Shapiro sequence, Hausdorff dimension, Diophantine equation, Diophantine approximation}

\begin{abstract}
A sequence of integers of the form $\lfloor n^{\alpha}\rfloor$ $(n=1,2,\ldots)$ for some fixed non-integral $\alpha>1$ is called a Piatetski-Shapiro sequence, where $\lfloor x\rfloor$ denotes the integer part of $x$. Let $\mathrm{PS}(\alpha)$ denote the set of all those terms. In this article, we show that $x+y=z$ has only finitely many solutions $(x,y,z)\in \mathrm{PS}(\alpha)^3$ for almost every $\alpha>3$. Furthermore, we show that $\mathrm{PS}(\alpha)$ has only finitely many arithmetic progressions of length $3$ for almost every $\alpha>10$. In addition, we estimate upper bounds for the Hausdorff dimension of the set of $\alpha\in [s,t]$ such that  $y=a_1x_1+\cdots +a_nx_n$ has infinitely many solutions on $\PS(\alpha)$. 
\end{abstract}

\maketitle

\section{Introduction}
Let $\mathbb{N}$ be the set of all positive integers. For all $x\in \mathbb{R}$, let $\lfloor x\rfloor$ be the integer part of $x$. A sequence $(a_n)_{n\in \mathbb{N}}$ of integers is called a \textit{Piatetski-Shapiro sequence} if there exists non-integral $\alpha>1$ such that $a_n=\lfloor n^\alpha \rfloor$ for every $n\in \mathbb{N}$. We set $\PS(\alpha)=\{\lfloor n^\alpha \rfloor \colon n\in \mathbb{N} \}$. In this article, for all fixed positive real numbers  $a_1,\ldots, a_n$, we investigate the finiteness of solutions $(y,x_1,\ldots, x_n)\in \PS(\alpha)^{n+1}$ to the following linear equation:
\begin{equation}\label{Equation-mainEq}
 y=a_1x_1+\cdots +a_nx_n.
\end{equation}
For simplicity, we first focus on the equation 
\begin{equation}\label{Equation-Fermat} 
x+y=z,
\end{equation}
where $(x,y,z)\in \PS(\alpha)^3$. If $\alpha=2$, then a tuple $(x,y,z)\in \PS(\alpha)^3$ satisfying \eqref{Equation-Fermat} is called a \textit{Pythagorean tuple}. There are infinitely many Pythagorean tuples since 
\[
(a^2-b^2)^2 +(2ab)^2=(a^2+b^2)^2
\]
holds for all $a,b\in \mathbb{N}$. If $\alpha\geq 3$ is an integer, then by Fermat's last theorem (the Fermat-Wiles theorem \cite[Theorem~0.5]{Wiles}), the equation \eqref{Equation-Fermat} does not have any solutions $(x,y,z)\in \PS(\alpha)^3$. Thus, if $\alpha \in \mathbb{N}$, then we can find a threshold at $\alpha=2$ or $3$ for the infiniteness or non-existence of solutions of \eqref{Equation-Fermat}. One of the motivations of this research is to extend the range of $\alpha$ from $\mathbb{N}$ to $\mathbb{R}$ and determine a threshold.  

Glasscock first began with the study of linear Diophantine equations on Piatetski-Shapiro sequences \cite{Glasscock17}. He investigated the equation
\begin{equation}\label{Equation-Glasscock}
y=ax+b
\end{equation}   
for fixed $a,b\in \mathbb{R}$ with $a\notin \{0,1\}$. Assume that \eqref{Equation-Glasscock} has infinitely many solutions $(x,y)\in \mathbb{N}^2$. In \cite{Glasscock17,Glasscock20}, he showed that 
for almost every $\alpha>1$ in the sense of one-dimensional Lebesgue measure,
\begin{itemize}
\item if $\alpha<2$, then \eqref{Equation-Glasscock} has infinitely many solutions $(x,y)\in \PS(\alpha)^2$;
\item if $\alpha>2$, then \eqref{Equation-Glasscock} has only finitely many solutions $(x,y)\in \PS(\alpha)^2$. 
\end{itemize} 
Thus, Glasssock discovered that $\alpha=2$ is a threshold for the infiniteness or finiteness of solutions of \eqref{Equation-Glasscock}. Furthermore, he proposed the following interesting question. 
\begin{question}[{\cite[Question~6]{Glasscock17}}]\label{Question-Glasscock}
Does there exist $\mathscr{G}>1$ with the property that for almost every or all $\alpha> 1$, the equation $x + y = z$ has infinitely many solutions $(x,y,z)\in \PS(\alpha)^3$ or not, according as $\alpha<\mathscr{G}$ or $\alpha>\mathscr{G}$?
\end{question}

By the result given by Frantzikinakis and Wierdl \cite[Proposition~5.1]{FrantzikinakisWierdl}, for all $1<\alpha<2$, the equation \eqref{Equation-Fermat} has infinitely many solutions $(x,y,z)\in \PS(\alpha)^3$. Further, Yoshida \cite{Yoshida} gave a quantitative result on the number of solutions  \eqref{Equation-Fermat}. For all $\alpha\in (1,6/5)$, he \cite[Corollary~1.3]{Yoshida} showed 
\begin{equation}\label{eq-Yoshida}
\lim_{x\to \infty}\frac{\# \{(\ell, m,n)\in \mathbb{N}^3 \colon n <x , \lfloor \ell^\alpha \rfloor + \lfloor m^\alpha \rfloor=\lfloor n^\alpha \rfloor \}}{x^{3-\alpha}}= \frac{\Gamma(1+1/\alpha)^2}{(3-\alpha)\Gamma(2/\alpha)}, 
\end{equation}
where $\Gamma(\cdot)$ denotes the Gamma function. Matsusaka and the author \cite{MatsusakaSaito} investigated the case when $\alpha$ is large. They \cite[Corollary~1.2]{MatsusakaSaito} showed that for every $s,t\in \mathbb{R}$ with $2<s<t$, there are uncountably many $\alpha\in [s,t]$ such that \eqref{Equation-Fermat} has infinitely many solutions $(x,y,z)\in \PS(\alpha)^3$ . Therefore, a threshold does not exist in the sense of the ``all" version of Question~\ref{Question-Glasscock}. However, we do not know whether a threshold exists or not in the sense of the ``almost every" version. In this article, we give a partial result of this problem.

\begin{theorem}\label{Theorem-main1}For almost every $\alpha>3$, the equation $x+y=z$ has only finitely many solutions $(x,y,z)\in \PS(\alpha)^3$.
\end{theorem}  

Assume that there exists a real number $\mathscr{G}\ge 1$ such that for almost every $\alpha> 1$, the equation \eqref{Equation-Fermat} has infinitely many solutions $(x,y,z)\in \PS(\alpha)^3$ or not, according as $\alpha<\mathscr{G}$ or $\alpha>\mathscr{G}$. We have $\mathscr{G}\geq 2$ since \eqref{Equation-Fermat} has infinitely many solutions for every $\alpha\in (1, 2)$ as we mentioned above. Furthermore, Theorem~\ref{Theorem-main1} gives the upper bound for $\mathscr{G}$ as $\mathscr{G}\leq 3$. Combining these results, we have 
\[
2\leq \mathscr{G}\leq 3. 
\]
However, we cannot verify the existence or determine the exact value of $\mathscr{G}$. The author conjectures that $\mathscr{G}$ exists and $\mathscr{G}=3$. Yoshida \cite[Conjecture~4.1]{Yoshida} conjectures that \eqref{eq-Yoshida} is valid for $\alpha\in (1,2)\cup (2,3)$. Furthermore, in  \cite[Section~4]{Yoshida}, he proposed a heuristic calculation of the problem. If this conjecture was true, then $\mathscr{G}$ would exist and $\mathscr{G}=3$. 

We will give a proof of Theorem~\ref{Theorem-main1} in Section~\ref{Section-SimpleCase}. We also reveal the following result on arithmetic progressions as follows.
\begin{theorem}\label{Theorem-main2}
For almost every $\alpha >10$, the set $\PS(\alpha)$ contains only finitely many arithmetic progressions of length $3$. 
\end{theorem}
By the result of Frantzikinakis and Wierdl \cite[Proposition~5.1]{FrantzikinakisWierdl}, the set $\PS(\alpha)$ contains arbitrarily long arithmetic progressions (APs) for all $\alpha\in (1, 2)$. In addition, the author and Yoshida showed that for all $1\leq \alpha<2$ and for any $B\subseteq \mathbb{N}$ with positive upper density (\textit{i.e.} $\limsup_{N\to \infty} \#(B\cap [1,N])/N>0 $), the set $\{\lfloor n^\alpha \rfloor \colon n\in B \}$ contains arbitrarily long APs \cite[Corollary~5]{SaitoYoshida}. Thus, they gave a Szemer\'edi-type theorem on $\PS(\alpha)$. Remark that Szemer\'edi's theorem states that any $B\subseteq \mathbb{N}$ with positive upper density contains arbitrarily long APs \cite[p.200]{Szemeredi}. On the other hand, the case $\alpha>2$ is not easy.  If $\alpha>2$ is an integer, then Darmon and Merel \cite[Main Theorem]{DarmonMerel} proved that $\PS(\alpha)$ does not contain any APs of length $3$. Moreover, Matsusaka and the author \cite[Corollary~1.3]{MatsusakaSaito} showed that for every $s,t\in \mathbb{R}$ with $2<s<t$, there are uncountably many $\alpha\in [s,t]$ such that $\PS(\alpha)$ contains infinitely many APs of length $3$. However, there was no result on the non-existence or finiteness of APs of length $3$ of $\PS(\alpha)$ if $\alpha>2$ is non-integral. Therefore, the main contribution of Theorem~\ref{Theorem-main2} is to establish the finiteness of such APs for almost every $\alpha>10$. However, we do not determine the infiniteness or finiteness of such APs for almost every $\alpha\in (2,10)$.

\begin{notation}
For all intervals $I\subset \mathbb{R}$, let  $I_\mathbb{Z}=I\cap \mathbb{Z}$. Further, we write $[N]=[1,N]_\mathbb{Z}$ for every $N\in \mathbb{N}$. Let $\mathcal{L}(F)$ denote the one-dimensional Lebesgue outer measure of $F$ for all $F\subseteq \mathbb{R}$. Let $P(x)$ be a statement depending on $x\in B$, where $B\subseteq \mathbb{R}$ is a Borel set. Then we say that for almost every $x\in U$, $P(x)$ is true if $\mathcal{L}(\{x\in B \colon \text{$P(x)$ is not true} \})=0$. For all $\alpha\in \mathbb{R}$ and $x>1$, we define $\log^\alpha x=(\log x)^\alpha$.   
\end{notation}

\section{General Results}

Let $\alpha>1$ be a non-integral real number, and fix positive real numbers $a_1,\ldots, a_n$. If $a_1+\cdots +a_n=1$, then each tuple $(y,x_1,\ldots, x_n)\in \PS(\alpha)^{n+1}$ with $y=x_1=\cdots =x_n$ is a solution to \eqref{Equation-mainEq}. Such a tuple is called a \textit{trivial} solution of \eqref{Equation-mainEq}. Other solutions are called \textit{non-trivial} solutions. If $a_1+\cdots +a_n\neq 1$, then all tuples $(y,x_1,\ldots, x_n)\in \PS(\alpha)^{n+1}$ satisfying \eqref{Equation-mainEq} are called \textit{non-trivial} solutions. The main goal of this article is to show the following theorems. Before stating the theorems, let $\Haus F$ denote the \textit{Hausdorff dimension} of $F\subseteq \mathbb{R}$ of which the precise definition will be given in Section~\ref{Section-SimpleCase}.   

\begin{theorem}\label{Theorem-genMain1}
Let $n\in \mathbb{N}$, and let $a_1,\ldots ,a_n$ be positive real numbers. Suppose that $a_i\geq 1$ for all $1\leq i \leq n $, or $a_1+\cdots +a_n<1$. Then for all $n+1<s<t$, we have
\begin{align*}
&\Haus 
\left\{\alpha\in [s,t]\ \colon\  
\begin{aligned}
&\text{\eqref{Equation-mainEq} has infinitely many non-trivial } \\ 
&\text{solutions } (y, x_1,\ldots ,x_n)\in \PS(\alpha)^{n+1}
\end{aligned} \right\}\leq  (n+1)/s.  
\end{align*}
In addition,  for almost every $\alpha>n+1$, the equation \eqref{Equation-mainEq} has only finitely many non-trivial solutions $(y,x_1,\ldots, x_n)\in \PS(\alpha)^{n+1}$. 
\end{theorem}

Further, in a general case, we obtain the following upper bounds.

\begin{theorem}\label{Theorem-genMain2}
Let $n\in \mathbb{N}$, and let $a_1,\ldots ,a_n$ be positive real numbers. Then for all $3n+4<s<t$, we have
\begin{align}\label{Ineuqality-genMain1}
&\Haus \left\{\alpha\in [s,t]\ \colon\  
\begin{aligned}
&\text{\eqref{Equation-mainEq} has infinitely many non-trivial } \\ 
&\text{solutions } (y, x_1,\ldots ,x_n)\in \PS(\alpha)^{n+1}
\end{aligned} \right\}\leq 3n/(s-4). 
\end{align}
In particular, for almost every $\alpha>3n+4$, the equation \eqref{Equation-mainEq} has only finitely many non-trivial solutions $(y,x_1,\ldots, x_n)\in \PS(\alpha)^{n+1}$. 

\end{theorem}

Let $a,b$ be real numbers with $a\neq 1$ and $0\leq b<a$. Suppose that 
\begin{equation}\label{Equation-Linear}
y=ax+b 
\end{equation}
has infinitely many solutions $(x,y)\in \mathbb{N}^2$. In \cite[Theorem~1.1]{Saito22}, for every $2<s<t$, the author showed
\begin{align*}
\Haus \left\{\alpha\in [s,t]\colon \text{\eqref{Equation-Linear} has infinitely many solutions $(x,y)\in \PS(\alpha)^{2}$}\right\}= 2/s.
\end{align*}
Thus, in the case $n=1$, the left-hand side of \eqref{Ineuqality-genMain1} is equal to $2/s$ if $a_1\neq 1$ is a positive rational number. In view of this result, the author conjectures that the left-hand side of \eqref{Ineuqality-genMain1} is equal to $(n+1)/s$ under suitable conditions on $a_1,\ldots, a_n$.

Theorems~\ref{Theorem-genMain1} and \ref{Theorem-genMain2} imply Theorems~\ref{Theorem-main1} and \ref{Theorem-main2}, respectively. Indeed, Theorem~\ref{Theorem-genMain1} with $n=2$ and $a_1=a_2=1$ implies Theorem~\ref{Theorem-main1}. Theorem~\ref{Theorem-genMain2} with $n=2$ and $a_1=a_2=1/2$ implies Theorem~\ref{Theorem-main2} since non-trivial solutions $(y,x_1,x_2)$ to \eqref{Equation-mainEq} are APs of length $3$ in this case.

\section{Hausdorff dimension}\label{Section-SimpleCase}
For all non-empty sets $X\subseteq \mathbb{R}$, we define $\diam X=\sup_{x,y\in X} |x-y|$ and set $\diam \emptyset =0$. Let $F\subseteq \mathbb{R}$. A family $\mathcal{U}$ of subsets of $\mathbb{R}$ is called a \textit{covering} of $F$ if  $F\subseteq \bigcup_{U\in \mathcal{U}}U$ and $\mathcal{U}$ is non-empty and at most countable. For all $\sigma>0$, we define
\[
\mathcal{H}_\infty ^{\sigma}(F) =\inf\left\{\sum_{U\in \mathcal{U}}(\diam U)^{\sigma} \colon \text{$\mathcal{U}$ is a covering of $F$}    \right\}.  
\]
Further, we define the \textit{Hausdorff dimension of $F$} as 
\begin{equation*}
\Haus F=\inf \{\sigma>0 \colon \mathcal{H}_\infty^\sigma (F)=0\}.  
\end{equation*}
This definition can be seen in \cite[p.51]{Falconer} produced by Falconer. 
We will apply the following basic properties of the Hausdorff dimension (see \cite[pp.48-49, (3.4)]{Falconer}):
\begin{enumerate} \renewcommand{\theenumi}{H\arabic{enumi}}
\renewcommand{\labelenumi}{(\theenumi)}    \setlength{\leftskip}{20pt}
\item \label{H1}for all $F\subseteq E\subseteq \mathbb{R}$, $\Haus F \leq \Haus E$;
\item \label{H2}for all $F_1,F_2,\ldots \subseteq \mathbb{R}$, $\Haus \bigcup_{j=1}^\infty F_j=\sup_{j=1,2,\ldots} (\Haus F_j)$;
\item \label{H3}if $\Haus F<1$, then $\mathcal{L}(F)=0$. 
\end{enumerate}
In order to obtain upper bounds for the Hausdorff dimension, it is useful to find a covering $\mathcal{U}$ of $F$. Indeed, let $\sigma$ be a suitable positive real number. For every $\epsilon>0$, if we can construct a covering $\mathcal{U}$ of $F$ satisfying 
\begin{equation*}
\sum_{U\in \mathcal{U}} (\diam U)^\sigma<\epsilon,
\end{equation*}
then $\mathcal{H}_\infty^\sigma(F) \leq \sum_{U\in \mathcal{U}} (\diam U)^\sigma<\epsilon$ by the definition of $\mathcal{H}^\sigma_\infty(\cdot)$. This implies $\mathcal{H}_\infty^\sigma(F)=0$. Therefore, we obtain $\Haus F\leq \sigma$ by the definition of the Hausdorff dimension. Thus, the main strategy of the proofs is to find a good covering $\mathcal{U}$ and evaluate $\diam U$ for $U\in \mathcal{U}$.

Before starting the proofs, let us prepare the $O$-notation and Vinogradov symbol. We write $O(1)$ for a bounded quantity. If this bound depends only on some parameters $a_1,\ldots, a_n$, then for instance we write $O_{a_1,a_2,\ldots, a_n}(1)$. As is customary, we often abbreviate $O(1)X$ and $O_{a_1,\ldots, a_n}(1)X$ to $O(X)$ and $O_{a_1,\ldots, a_n}(X)$ respectively for a non-negative quantity $X$. We also say $f(X) \ll g(X)$ and $f(X) \ll_{a_1,\ldots, a_n} g(X)$  as $f(X)=O(g(X))$ and $f(X)=O_{a_1,\ldots, a_n}(g(X))$ respectively, where $g(X)$ is non-negative. 

\begin{lemma}\label{Lemma-logevalution}
Let $\sigma>0$. For all $r\geq 2$, we have 
\[
\sum_{1\leq q <r}  \log^{-\sigma} (r/q)\ll_\sigma (r+r^\sigma)\log r.
\]
\end{lemma}

\begin{proof}
Since $f(u)=\log^{-\sigma} (r/u)$ is continuously increasing on $u\in [1,r-1]$, we have
\begin{align*}
\sum_{1\leq q <r}  \log^{-\sigma} (r/q)
&\leq \int_{1}^{r-1} \log^{-\sigma} (r/u) du +   \log^{-\sigma}\left(\frac{r}{r-1}\right)\\
&=\int_{1}^{r-1} \log^{-\sigma} (r/u) du+ O(r^\sigma).
\end{align*}
Furthermore, by changing variables with $x=\log (r/u)$, we obtain
\[
\int_{1}^{r-1} \log^{-\sigma} (r/u) du= r \int_{\log \frac{r}{r-1}}^{\log r} x^{-\sigma} e^{-x} dx \ll_\sigma  (r+r^\sigma)\log r,    
\]
which completes the lemma.
\end{proof}

\section{Construction of Coverings}\label{Section-General}
Let $a_1,\ldots, a_n$ be positive real numbers. For all $1<s<t$,  we define
\begin{equation*}
\mathcal{A}(s,t)=\mathcal{A}_{a_1,\ldots, a_n}(s,t)= \left\{\alpha\in [s,t]\ \colon\  
\begin{aligned}
&\text{\eqref{Equation-mainEq} has infinitely many non-trivial } \\ 
&\text{solutions } (y, x_1,\ldots ,x_n)\in \PS(\alpha)^{n+1}
\end{aligned} \right\}
\end{equation*}
By the definition of non-trivial solutions, we reform the above as 
\begin{equation*}
\mathcal{A}_{a_1,\ldots, a_n}(s,t)= \left\{\alpha\in [s,t]\ \colon\  
\begin{aligned}
&\text{\eqref{Equation-mainEq} has infinitely many solutions } \\ 
&(y, x_1,\ldots ,x_n)\in \PS(\alpha)^{n+1}\text{ with }\#\{y,x_1,\ldots,x_n\}>1
\end{aligned} \right\}.
\end{equation*}
Also, we define
\begin{align*}
\mathcal{B}(s,t)&= \mathcal{B}_{a_1,\ldots, a_n}(s,t)\\
&= \left\{\alpha\in [s,t]\ \colon\  
\begin{aligned}
&\text{\eqref{Equation-mainEq} has infinitely many solutions } \\ 
&(y, x_1,\ldots ,x_n)\in \PS(\alpha)^{n+1}\text{ with }\#\{y,x_1,\ldots,x_n\}= n+1
\end{aligned} \right\}
\end{align*}

\begin{lemma}\label{Lemma-AtoB1}For all $1<s<t$, we have
\begin{equation}\label{Equation-AtoB1}
\Haus \mathcal{A}_{a_1,\ldots, a_n}(s,t) \leq \sup_{\substack{b_1,\ldots, b_k>0 \\ k\in [n]}} \Haus \mathcal{B}_{b_1,\ldots, b_k}(s,t). 
\end{equation}
\end{lemma}
\begin{proof} If $n=1$, then \eqref{Equation-AtoB1} is trivial. We may assume that $n\geq 2$. Then, we set $\mathcal{F}=\mathcal{A}(s,t)\setminus \mathcal{B}(s,t)$. If $\mathcal{F}$ is empty, then \eqref{Equation-AtoB1} is trivial. We may assume that $\mathcal{F}$ is non-empty. We take an arbitrary $\alpha\in \mathcal{F}$. Then  \eqref{Equation-mainEq} holds for infinitely many $(y,x_1,\ldots, x_n)\in\PS(\alpha)^{n+1}$ with $1< \#\{y,x_1,\ldots,x_n\}<n+1$. Take such a solution $(y,x_1,\ldots, x_n)$. Then, the following \eqref{itemi} or \eqref{itemii} is true:  
\begin{enumerate} \renewcommand{\theenumi}{\roman{enumi}}
\renewcommand{\labelenumi}{(\theenumi)}
\item \label{itemi} $y=x_i$ for some $i\in[n]$; 
\item \label{itemii} $x_i=x_j$ for some $i\neq j$.
\end{enumerate}
If \eqref{itemi} is true, then we see that $y =a_1 x_1 +\cdots +a_nx_n > a_i y$, 
which implies $a_i<1$. Thus, we have $1-a_i>0$ and 
\begin{equation}\label{Equation-newone}
y= \frac{a_1}{1-a_i} x_1 +\cdots + \frac{a_{i-1}}{1-a_i}x_{i-1}+\frac{a_{i+1}}{1-a_i}x_{i+1}+\cdots+\frac{a_n}{1-a_i}x_n.    
\end{equation}
The number of the variables $x_k$ is $n-1$ in the equation \eqref{Equation-newone}, and all of its coefficients are positive. If \eqref{itemii} is true, then we also obtain a linear Diophantine equation with $n-1$ variables $x_k$ and positive coefficients.

Therefore, there exists a  finite set $V\subset \mathbb{R}^{n-1}$ such that
\begin{align*}\label{Equaiton-union1} 
\mathcal{A}_{a_1,\ldots, a_n}(s,t)&=\mathcal{B}_{a_1,\ldots, a_n}(s,t) \cup\mathcal{F}\\
&\subseteq \mathcal{B}_{a_1,\ldots, a_n}(s,t) \cup \bigcup_{ \substack {(b_1,\ldots ,b_{n-1}) \in V\\ b_1,\ldots ,b_{n-1}>0 }} \mathcal{A}_{b_1,\ldots, b_{n-1}}(s,t).
\end{align*}
By iterating a similar manner for $\mathcal{A}_{b_1,\ldots, b_{n-1}}(s,t)$,  there exists a finite set $V'\subset \bigcup_{k=1}^n \mathbb{R}^k$ such that 
\[
\mathcal{A}_{a_1,\ldots, a_n}(s,t)\subseteq \bigcup_{\substack{(b_1,\ldots, b_k)\in V' \\ k\in[n] }}   \mathcal{B}_{b_1,\ldots, b_k}(s,t).
\]  
By \eqref{H1}, \eqref{H2} in Section~\ref{Section-SimpleCase}, we conclude Lemma~\ref{Equation-AtoB1}. 
\end{proof}

\begin{lemma}\label{Lemma-AtoB2}
Suppose that $a_i\geq 1$ for all $i\in [n]$. Then for all $1<s<t$, we have
\[
\Haus \mathcal{A}_{a_1,\ldots, a_n}(s,t) \leq \sup_{\substack{b_1,\ldots, b_k\geq 1,\\   k\in [n]}} \Haus \mathcal{B}_{b_1,\ldots, b_k}(s,t). 
\]
\end{lemma}
\begin{proof}
Let $\mathcal{F}$ be as in the proof of Lemma~\ref{Lemma-AtoB1}. Similarly with the proof of Lemma~\ref{Lemma-AtoB1}, we may assume that $n\geq 2$ and $\mathcal{F}$ is non-empty. Take an arbitrary $\alpha\in \mathcal{F}$. Since $a_i\geq 1$ for all $i\in [n]$, the equation \eqref{Equation-mainEq} does not have any non-trivial solutions $(y,x_1,\ldots, x_n)\in \PS(\alpha)^{n+1}$ with $y=x_i$ for some $i\in [n]$. Therefore, \eqref{itemii} in the proof of Lemma~\ref{Lemma-AtoB1} only happens. Since  $a_i+a_j\geq 1$ holds for all $1\leq i<j\leq n$, similarly with the proof of Lemma~\ref{Lemma-AtoB1}, we obtain Lemma~\ref{Lemma-AtoB2}.
\end{proof}

\begin{lemma}\label{Lemma-AtoB3}Suppose that $a_1+\cdots+a_n<1$. Then for all $1<s<t$, we have
\begin{equation*}
\Haus \mathcal{A}_{a_1,\ldots, a_n}(s,t) \leq \sup_{\substack{b_1,\ldots, b_k>0, \\ b_1+\cdots+b_k<1, \\ k\in [n]}} \Haus \mathcal{B}_{b_1,\ldots, b_k}(s,t).
\end{equation*}
\end{lemma}

\begin{proof}
Since $a_i<1$ for all $i\in [n]$ and  $a_1+\cdots+a_n<1$, we have
\[
\frac{a_1}{1-a_i}  +\cdots + \frac{a_{i-1}}{1-a_i}+\frac{a_{i+1}}{1-a_i}+\cdots+\frac{a_n}{1-a_i}< \frac{1-a_i}{1-a_i}=1.
\]
Therefore, we similarly obtain Lemma~\ref{Lemma-AtoB3} to the proof of Lemma~\ref{Lemma-AtoB1}. 
\end{proof}

By the above lemmas, it suffices to estimate upper bounds for $\Haus \mathcal{B}_{b_1,\ldots, b_k}(s,t)$. \\

Fix arbitrary $k\in [n]$ and real numbers $b_1,\ldots,b_k>0$. Let $1<\beta<s<t<\gamma$. Set $\mathcal{B}(s,t)=\mathcal{B}_{b_1,\ldots , b_k}(s,t)$. We do not indicate the dependencies on the numbers $b_1,\ldots, b_k$, $k$, $n$, $\beta$, $s$, $t$, $\gamma$. We consider these numbers as absolute constants. Let $M_0$ be a sufficiently large real number. For all $u\in \mathbb{R}$ and $Q_1,\ldots, Q_k>0$, we define 
\begin{equation}\label{equation-defE}
E (u)=E(u;Q_1,\ldots, Q_k)=b_1 Q_1^u +\cdots +b_kQ_k^u.
\end{equation}
For all $r, q_1,\ldots ,q_k\in \mathbb{N}$, we define
\begin{gather*}
q=(q_1,\ldots,q_k),\quad q/r=(q_1/r,\ldots, q_k/r), \quad Q_j=q_j/r , \quad Q=(Q_1,\ldots, Q_k)
\end{gather*}
and 
\begin{align*}
J(q;r)&=E^{-1}([1-r^{-\beta},1+r^{-\beta} ]; q/r ) \cap [s,t]  \\ 
&=\{\alpha\in [s,t]\colon E(\alpha; q/r )\in [1-r^{-\beta},1+r^{-\beta}] \}.
\end{align*}

\begin{lemma}\label{Lemma-B(M)}
For every $M\geq M_0$, we have
\begin{equation}\label{Inequality-Maincovering}
\mathcal{B}(s,t) \subseteq \bigcup_{r\geq M}  \bigcup_{\substack{1\leq q_j < B_j r \\ q_j \neq r, \ \forall j\in [k] } }  J(q ; r),
\end{equation}
where $B_j=\max(b_j^{-1/\beta},\ b_j^{-1/\gamma} )$ for every $j\in [k]$. 
\end{lemma}

\begin{proof}
Fix an arbitrary $\alpha\in \mathcal{B}(s,t)$. Take any $M\geq M_0$. Then, there are infinitely many $(r,q_1,\ldots, q_k)\in \mathbb{N}^{k+1}$ with $\#\{r,q_1,\ldots, q_k\}=k+1$ and $r\geq M$ such that 
\begin{equation}\label{Equation-rbq}
\lfloor r^\alpha \rfloor = b_1\lfloor q_1^\alpha \rfloor +\cdots +b_k \lfloor q_k^\alpha \rfloor. 
\end{equation}
Take any $j\in [k]$. Then, we obtain 
\[
b_j(q_j^\alpha-1) < b_1\lfloor q_1^\alpha \rfloor +\cdots +b_k \lfloor q_k^\alpha \rfloor=\lfloor r^\alpha \rfloor\leq r^\alpha, 
\]
which implies that
\begin{equation}\label{Inequality-qjbj}
q_j< (b_j^{-1} r^\alpha +1)^{1/\alpha}\leq b_j^{-1/\alpha} r+1.
\end{equation}
If $b_j=1$, then $B_j=\max(b_j^{-1/\beta},\ b_j^{-1/\gamma} )=1$, and hence $q_j<r=B_j r$ by \eqref{Inequality-qjbj} and $q_j\neq r $. If $b_j\neq 1$, then we have $q_j<B_j r$ by \eqref{Inequality-qjbj} and $r\geq M\geq M_0$. 

In addition, by \eqref{Equation-rbq}, we have
\[
\left|b_1\left(\frac{ q_1}{ r}\right)^\alpha +\cdots + b_k\left(\frac{ q_k}{ r}\right)^\alpha -1 \right| \ll \frac{1}{r^\alpha}.
\]
Therefore, $E(\alpha; q/r)\in [1-r^{-\beta},1+r^{-\beta}]$ and $\alpha\in [s,t]$ since $M_0$ is sufficiently large. Thus, $\alpha\in J(q; r)$ holds, and hence we conclude \eqref{Inequality-Maincovering}. 
\end{proof}

For all $r\in \mathbb{N}$ and $j\in [k]$, we define
\[
I_j^{(0)}=I_j^{(0)}(r)= [1,r)_{\mathbb{Z}}\quad \text{ and }\quad I_j^{(1)}=I_j^{(1)}(r)=(r,B_jr)_{\mathbb{Z}}.  
\]
For every $\nu=(\nu_1,\ldots,\nu_k)\in \{0,1\}^k$, we set 
\[
I^{\nu}=I_1^{(\nu_1)}\times \cdots \times I_k^{(\nu_k)}, 
\]
where $I_1^{(\nu_1)}\times \cdots \times I_k^{(\nu_k)}=\emptyset$ if $I_i^{(\nu_i)}=\emptyset$ for some $i\in [k]$. By Lemma~\ref{Lemma-B(M)}, we obtain
\begin{equation}\label{Inclusion-Covering}
\mathcal{B}(s,t) \subseteq \bigcup_{\nu\in \{0,1\}^k} \bigcup_{r\geq M}  \bigcup_{\substack{q\in I^{\nu} } }  J(q ; r),
\end{equation}  
We decompose the above union into $3$ cases as follows:
\begin{itemize} 
\item (Case~0) $\nu_1=\cdots =\nu_k=0$;
\item (Case~1) $\nu_1=\cdots =\nu_k=1$;
\item (Case~2) otherwise.
\end{itemize}
For all $\nu\in \{0,1\}^n$ and $M\geq M_0$, we set 
\begin{equation*}
\mathcal{C}^{\nu}(M)= \{J(q; r)\colon r\geq M,\  q\in I^\nu \}
\end{equation*}
Then we have 
\[
\bigcup_{\nu\in \{0,1\}^k} \mathcal{C}^{\nu}(M)=\mathcal{C}_0(M)\cup \mathcal{C}_1(M) \cup \mathcal{C}_2(M),
\]
 where 
\begin{gather*}
\mathcal{C}_0(M)=\mathcal{C}^{(0,\ldots, 0)}(M),\quad \mathcal{C}_1(M)=\mathcal{C}^{(1,\ldots, 1)}(M),\quad
\mathcal{C}_2(M)=\bigcup_{\substack{\nu\in \{0,1\}^k \\ \#\{\nu_1,\ldots, \nu_k\}>1 }} \mathcal{C}^{\nu} (M).  
\end{gather*}
By \eqref{Inclusion-Covering}, for every $M\geq M_0$, we obtain
\begin{equation}\label{Inclusion-cases}
\mathcal{B}(s,t) \subseteq \left(\bigcup_{J\in \calC_0(M)} J\right) \cup \left( \bigcup_{J\in \calC_1(M)} J \right)\cup \left(\bigcup_{J\in \calC_2(M)} J  \right).
\end{equation}

\begin{remark}\label{Remark-SimpleCase1} 
If $b_j\geq 1$ for all $j\in [k]$, then $B_j\leq 1$, and hence $I^{(1)}_j=\emptyset$ for all $j\in[k]$. Therefore, in this case,  for every $M\geq M_0$, we have 
\begin{equation}\label{Inequality-covering1}
\mathcal{B}(s,t) \subseteq  \bigcup_{J\in \calC_0(M)} J.
\end{equation}
\end{remark}

\section{Cases~0 and 1}

\begin{lemma}[Case~0]\label{Lemma-BasicCase1}
Let $\sigma> (k+1)/\beta$. For all $M\geq M_0$, we have
\begin{align*}
\sum_{J\in \calC_{0}(M)} (\diam J)^\sigma <\infty .
\end{align*}

\end{lemma}

\begin{proof}
Take any $q=(q_1,\ldots, q_k)\in I^{(0,\ldots ,0)}$ and $r\geq M$. Then $E(u)=E(u;q/r)$ is monotonically decreasing and continuously differentiable on $[0,\infty)$ since $1\leq q_1,\ldots, q_k <r$. Hence, the inverse function of $E$
\[
L(\ \cdot\ )=L(\ \cdot\ ;q/r): (0,\sum_{i\in [k]} b_i]\ni E(u) \longmapsto u\in [0,\infty) 
\]
exists. The function $L$ is also monotonically decreasing and continuously differentiable. By setting 
 $A= [1-r^{-\beta}, 1+r^{-\beta}]\cap [E(s),E(t)]$, we obtain
\[
J(q; r)=E^{-1} ([1-r^{-\beta}, 1+r^{-\beta}]; q/r )\cap [s,t] = L(A)
\]
since $L$ is injective. Hence, by the mean value theorem, there exists $\theta\in A$ such that
\[
\diam J(q; r) \leq L(1-r^{-\beta})-L(1+r^{-\beta})= -2r^{-\beta} L'(\theta)=2r^{-\beta}\frac{1}{-E'(L (\theta) ) }.  
\]
Let $\max q= \max (q_1,\ldots ,q_k)$. By $\theta\in A\subseteq [1-r^{-\beta}, 1+r^{-\beta}]$, we obtain
\begin{align*}
-E'(L (\theta) ) &= b_1 (q_1/r)^{L(\theta)} \log (r/q_1) +\cdots + b_k (q_k/r)^{L(\theta)} \log (r/q_k)\\
&\geq E(L(\theta)) \log(r/\max q)\gg \log(r/\max q).
\end{align*}
Therefore, we have $\diam J(q; r)\ll r^{-\beta}/\log (r/\max q )$, which leads to
\begin{align*}
\sum_{J\in \calC_{0}(M)} (\diam J)^\sigma &= \sum_{r\geq M} \sum_{q\in I^{(0,\ldots, 0)}} (\diam J(q;r))^\sigma \\
&\ll   \sum_{r\geq M} \sum_{q\in I^{(0,\ldots, 0)}} \left(\frac{r^{-\beta}}{\log (r/\max q)} \right)^\sigma\\
&\ll \sum_{r\geq M} r^{-\beta \sigma +k-1} \sum_{q_k\in [r-1]} (\log (r/q_k))^{-\sigma}.
 \end{align*}
By Lemma~\ref{Lemma-logevalution}, the most right-hand side of the above inequality is $ \ll\sum_{r\geq M} r^{-\beta \sigma +k}\log r$. Since this series converges if $-\beta \sigma+k<-1$, we complete the lemma.
\end{proof}

\begin{lemma}[Case~1]\label{Lemma-BasicCase2}
Let $\sigma> (k+1)/\beta$. For all $M\geq M_0$, we have
\begin{align*}
\sum_{J\in \calC_{1}(M)} (\diam J)^\sigma <\infty .
\end{align*}
\end{lemma}

\begin{proof}We omit the proof because it is similar to Lemma~\ref{Lemma-BasicCase1}. 
\end{proof}

By combining Lemmas~\ref{Lemma-BasicCase1} and \ref{Lemma-BasicCase2},  if $\beta>k+1$, then each $\sigma \in ((k+1)/\beta, 1)$ satisfies
\begin{equation}\label{Series-Case1Case2}
\sum_{J\in \calC_0 (M)\cup \calC_1 (M)}  \left( \diam J\right)^ \sigma<\infty.
\end{equation}
Therefore, the left-hand side of \eqref{Series-Case1Case2} goes to $0$ as $M\to \infty$. 

\section{Case~2 with \texorpdfstring{$b_1+\cdots + b_k<1$}{the sum of b's less than 1}}\label{Section-Case3smallS}
Fix any $\nu=(\nu_1,\ldots,\nu_k)\in \{0,1\}^k$ with $\#\{\nu_1,\ldots, \nu_k\}>1$. We assume that  $I_{1}^{(\nu_1)}$,\ldots, $I_{k}^{(\nu_k)}$ are non-empty. Let $M_0$ be a sufficiently large positive number. Suppose that $M\geq M_0$, $r\geq M$, and $q=(q_1,\ldots, q_k)\in I^\nu$. Recall that 
\[
Q_j=q_j/r\ (\text{for }j\in [k])\quad \text{and} \quad Q=(Q_1,\ldots, Q_k).
\]

\begin{lemma}\label{Lemma-SecondDerivE} For all $u>0$ and $j\in [k]$, we have
\begin{equation}\label{Inequality-SecondDerivE}
E''(u;Q)\gg  Q_j^u (\log Q_j)^2,
\end{equation}
where the implicit constant is absolute. Furthermore, if $b_1+\cdots +b_k<1$ and $J(q; r)$ is non-empty, then for all $u>0$
\begin{equation}\label{Inequality-lowerforsecond}
E''(u;Q)\gg 1,
\end{equation}
where the implicit constant is absolute.
\end{lemma}
\begin{proof} 
For all $u>0$ and $j\in [k]$, we have
\begin{align*}
&E''(u;Q)\geq b_1 Q_1^u (\log Q_1)^2 +\cdots + b_k Q_k^u (\log Q_k)^2 \geq b_j Q_j^u (\log Q_j)^2,
\end{align*}
which leads to \eqref{Inequality-SecondDerivE}.

 We next show \eqref{Inequality-lowerforsecond}. Let $S=b_1+\cdots+b_k$. Assume that $S<1$ and $J(q; r)\neq \emptyset$. Then, there exists $u'\in [s,t]$ such that 
 \begin{equation}\label{Inequality-lowerforsecond1}
 E(u')\in [1-r^{-\beta}, 1+r^{-\beta}] 
 \end{equation}
 by $J(q; r)\neq \emptyset$. Let $\delta$ be a sufficiently small positive real number depending only on $S$ such that $S(1+\delta)<1$. Let us show that $Q_j> (1+\delta)^{1/t}$ for some $j\in [k]$ by contradiction. Suppose that $Q_j\leq (1+\delta)^{1/t}$ for all $j\in [k]$. Then, we obtain
\[
E(u')\leq S(1+\delta)^{u'/t} \leq S(1+\delta)< 1-M_0^{-\beta}\leq 1-r^{-\beta},
\]
since $M_0$ is sufficiently large and $r\geq M\geq M_0$, a contradiction to \eqref{Inequality-lowerforsecond1}. Therefore, we have $Q_j >(1+\delta)^{1/t}$ for some $ j\in [k]$, and hence we conclude \eqref{Inequality-lowerforsecond}  by \eqref{Inequality-SecondDerivE}.  
\end{proof}

 By \eqref{equation-defE} and $\#\{\nu_1,\ldots, \nu_k\}>1$, we obtain 
\begin{gather*}
\lim_{u\to\infty} E(u;Q) =\lim_{u\to-\infty} E(u;Q) =\infty,\\
E''(u;Q) >0 \quad \text{for all $u\in \mathbb{R}$}.
\end{gather*}
Therefore, $E(u;Q)$ has the unique minimal point on $\mathbb{R}$. Let 
\[
m=m(Q)=\min_{u\in \mathbb{R}} E(u;Q), 
\]
and let $u_0=u_0(Q)\in \mathbb{R}$ be the point satisfying $m=E(u_0;Q)$. In Figure~\ref{Figure-Case2}, we plot the graph of $y=E(u;Q)$ when $E(0;Q)=b_1+\cdots + b_k<1$.

\begin{figure}[htbp]
\includegraphics[scale=0.05]{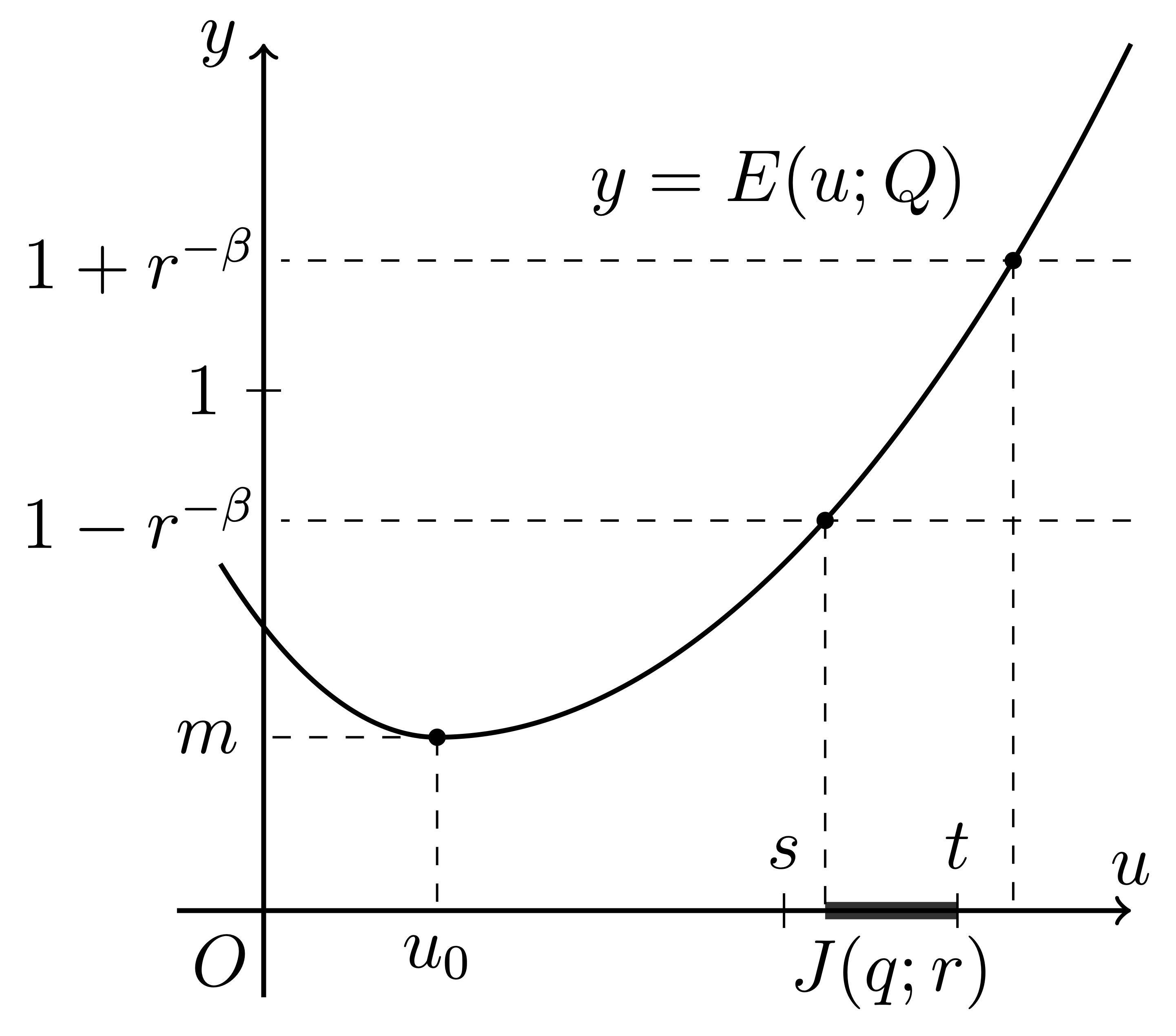}
\caption{Case~2 with $b_1+\cdots+b_k<1$}\label{Figure-Case2}
\end{figure}

\begin{lemma}\label{Lemma-smallSdiam}
If $b_1+\cdots+b_k<1$, then we have $\diam J(q; r) \ll r^{-\beta}$, where the implicit constant is absolute.
\end{lemma}
\begin{proof}
Let $S=b_1+\cdots+b_k<1$. We now observe that $J(q;r)\subseteq (u_0, \infty)$. Indeed, we take any $\alpha\in J(q;r)$. Then, we have $\alpha \in [s,t]$ and $E(\alpha)\in [1-r^{-\beta},1+r^{-\beta}]$. If $u_0\leq 0$, then it is clear that $\alpha\in (u_0,\infty)$. If $u_0>0$, then we show that $\alpha\in (u_0,\infty)$ by contradiction. Suppose that $\alpha\leq u_0$ is true. Then, we observe that
\begin{equation}\label{Inequality-smallS}
S=E(0)> E(\alpha)\geq 1-r^{-\beta}
\end{equation}
since $E(\cdot)$ is monotonically decreasing on $(0,u_0]$. By $1>S$ and $r\geq M_0$, the inequality \eqref{Inequality-smallS} leads to a contradiction since $M_0$ is sufficiently large. Therefore, we have $\alpha\in(u_0,\infty)$.

Since $E$ is monotonically increasing and continuously differentiable on $(u_0,\infty)$, the inverse function $L \colon (m,\infty)\to (u_0,\infty)$ of $E\mid_{(u_0,\infty)} $ exists. Here, $f\mid_A$ denotes the restricted function of $f$ on $A$. Since $J(q;r)\subseteq (u_0,\infty)$, we have
\begin{align*}
J(q; r)&=[s,t]\cap L ([1-r^{-\beta},1+r^{-\beta}]\cap (m,\infty)).
\end{align*}
If $J(q; r)$ is empty, then $\diam J(q; r)=0$. Thus, we may assume that  $J(q; r)$ is non-empty. In this case, by the mean value theorem, there exists $\theta\in [1-r^{-\beta},1+r^{-\beta}]\cap (m,\infty)$  such that $L(\theta)\in [s,t]$ and 
\[
\diam J(q; r) \leq 2r^{-\beta}L'(\theta)= 2 r^{-\beta} \frac{1}{E' (L(\theta) ) }. 
\]
If $u_0<0$, then $E'(0)>0$. In this case, by the mean value theorem and Lemma~\ref{Lemma-SecondDerivE}, there exists $\theta'$ in between $0$ and $L(\theta)$ such that
\[
E'(L(\theta))\geq E'(L(\theta))- E' (0)= |L(\theta)||E''(\theta') |\geq s|E''(\theta') |\gg 1 .
\]
Therefore, we conclude that $\diam J(q; r)\ll r^{-\beta}$ if $u_0<0$.

Suppose that $u_0\geq 0$. Then, by $E'(u_0)=0$ and Lemma~\ref{Lemma-SecondDerivE}, there exists $\tilde{\theta}\in (u_0, L(\theta))$ such that
\begin{equation}\label{Inequality-smallS1}
E'(L(\theta))= E'(L(\theta)) -E'(u_0) = (L(\theta)-u_0)E''(\tilde{\theta}) \gg |L(\theta)-u_0|.
\end{equation}
Further, by the mean value theorem, we find $\eta\in (u_0, L(\theta))$ such that
\begin{equation}\label{Inequality-smallS2}
E(L(\theta))-E(u_0)=|L(\theta)-u_0| E'(\eta).
\end{equation}
Recalling that $S<1$, $\theta\geq 1-r^{-\beta}$, and $m=E(u_0)=\min E(u)$, we have 
 \begin{equation}\label{Inequality-smallS3}
 E(L(\theta))-E(u_0)=\theta -m\geq 1-r^{-\beta}-S \geq (1-S)/2\gg 1
 \end{equation}
since $r\geq M\geq M_0$ and $M_0$ is sufficiently large. Thus, by combining \eqref{Inequality-smallS1}, \eqref{Inequality-smallS2}, and \eqref{Inequality-smallS3}, we have 
\[
E'(L(\theta))\gg |L(\theta)-u_0|=\frac{E(L(\theta))-E(u_0)}{E'(\eta)} \gg \frac{1}{E'(\eta)},
\]
where $E'(\eta)>0$ by $\eta>u_0$. We also recall that  $Q_j< B_j$ and $B_j=\max(b_j^{-1/\beta},\ b_j^{-1/\gamma} )$ for all $j\in [k]$, which implies that  
\[ 
E'(\eta)= b_1 Q_1^{\eta} \log Q_1+ \cdots + b_k Q_k^{\eta}\log Q_k    \leq k \max_{j\in [k]} (B_j^\eta \log B_j).
\]
Therefore, we obtain $E'(L(\theta))\gg 1$, and hence we conclude the lemma.
\end{proof}

\section{Proof of Theorem~{\ref{Theorem-genMain1}}}
Let $a_1,\ldots, a_n$ be positive real numbers. Let $n+1<\beta<s<t$. Suppose that $a_i\geq 1$ for all $i\in [k]$. Then, by Lemma~\ref{Lemma-AtoB2}, it suffices to show that 
\begin{equation}\label{Inequality-Desired01}
\Haus \mathcal{B}_{b_1,\ldots, b_k}(s,t)\leq (k+1)/s
\end{equation}
 for all $k\in [n]$ and $b_1,\ldots, b_k\geq 1$. Fix such $k\in [n]$ and $b_1,\ldots, b_k$. Let $\sigma\in ((k+1)/\beta,1)$. By \eqref{Inequality-covering1}, we obtain 
\[
\mathcal{H}^\sigma_\infty (\mathcal{B}_{b_1,\ldots, b_k}(s,t)) \ll \sum_{J\in \mathcal{C}_1(M)}  \left(\diam J\right)^\sigma
\]
which converges by Lemma~\ref{Lemma-BasicCase1} and $\sigma> (k+1)/\beta$.  Taking $M\to \infty$, we have
$\mathcal{H}^\sigma_\infty (\mathcal{B}(s,t))=0$, and hence $\Haus \mathcal{B}(s,t)\leq \sigma$. By choosing $\sigma\to (k+1)/\beta$ and $\beta\to s$, we conclude \eqref{Inequality-Desired01}. 

Suppose that $a_1+\cdots+a_n < 1$. Then, recalling Lemma~\ref{Lemma-AtoB3},  it suffices to show \eqref{Inequality-Desired01} for all $k\in [n]$ and $b_1,\ldots, b_k>0$ with $b_1+\cdots +b_k<1$. Fix arbitrary $k\in[n]$ and such $b_1,\ldots, b_k$. By \eqref{Inclusion-cases}, we have
\begin{align*}
&\mathcal{H}^\sigma_\infty (\mathcal{B}_{b_1,\ldots, b_k}(s,t))\\
&\leq \sum_{J\in \calC_0(M)} (\diam J)^\sigma + \sum_{J\in \calC_1(M)} (\diam J)^\sigma +  \sum_{\substack{\nu\in \{0,1\}^k \\ \#\{\nu_1,\ldots, \nu_k\}>1 }} \sum_{J\in \calC^\nu(M)} (\diam J)^\sigma  
\end{align*}
The first and second series converge by \eqref{Series-Case1Case2} and $\sigma>(k+1)/\beta$. Moreover, by Lemma~\ref{Lemma-smallSdiam}, the third series is 
\[
\ll_\sigma \sum_{\substack{\nu\in \{0,1\}^k \\ \#\{\nu_1,\ldots, \nu_k\}>1 }} \sum_{r\geq M} \sum_{q\in I^\nu} r^{-\sigma\beta}\ll \sum_{r\geq M} r^{-\beta \sigma + k}
\]
which also converges. Therefore, by taking $M\to \infty$, we have
$\mathcal{H}^\sigma_\infty (\mathcal{B}(s,t))=0$. Hence, $\Haus \mathcal{B}(s,t)\leq \sigma$.  By choosing $\sigma\to (k+1)/\beta$ and $\beta\to s$, we obtain \eqref{Inequality-Desired01}. Furthermore, $\Haus \mathcal{B}(s,t)<1$ implies $\mathcal{L}( \mathcal{B}(s,t))=0$ by \eqref{H3}. Therefore, we conclude Theorem~\ref{Theorem-genMain1}.

\section{Other cases}\label{Section-nonsimple}

Suppose that $b_1,\ldots, b_k$ are real numbers with $b_i >0$ for all $i\in [k]$.  Fix arbitrary $1<\beta <s<t<\gamma$. By recalling \eqref{Inclusion-cases}, we should estimate upper bounds for the diameter of $J$ for all $J\in  \mathcal{C}^{\nu} (M)$ and $\nu \in \{0,1\}^k$ with $\#\{\nu_1,\ldots, \nu_k\}>1$. Without loss of generality, we may assume that $\nu_k=1$ by renumbering.

Suppose that $M\geq M_0$ and $r\geq M$.  Then, $I^{(\nu_k)}_k= I^{(1)}_k= (r,B_k r)$ and $B_k>1$. If $P=P(q,r)$ is a condition depending on  $q$ and $r$, then we define 
\begin{equation*}
\mathcal{C}^\nu( M;P)= \{J(q; r)\colon r\geq M,\  q\in I^\nu,\ \text{$P(q,r)$ is true} \}.
\end{equation*}
  Let $u_0=u_0(\cdot)$  be as in Section~\ref{Section-Case3smallS}. We consider the following three cases:  
\begin{itemize}
\item (Case~2.0)  $u_0(q/r)\leq \beta$;  
\item (Case~2.1)  $u_0(q/r)\geq \gamma$; 
\item (Case~2.2)  $u_0(q/r)\in (\beta,\gamma)$.
\end{itemize}
In Figures~\ref{Figure-Case20}, \ref{Figure-Case21}, and \ref{Figure-Case22}, we plot the graph of $y=E(u;Q)(=E(u;q/r))$ in each case.

\begin{figure}[htbp]
\centering
\begin{minipage}[b]{0.32\columnwidth}
    \centering
    \includegraphics[width=0.9\columnwidth]{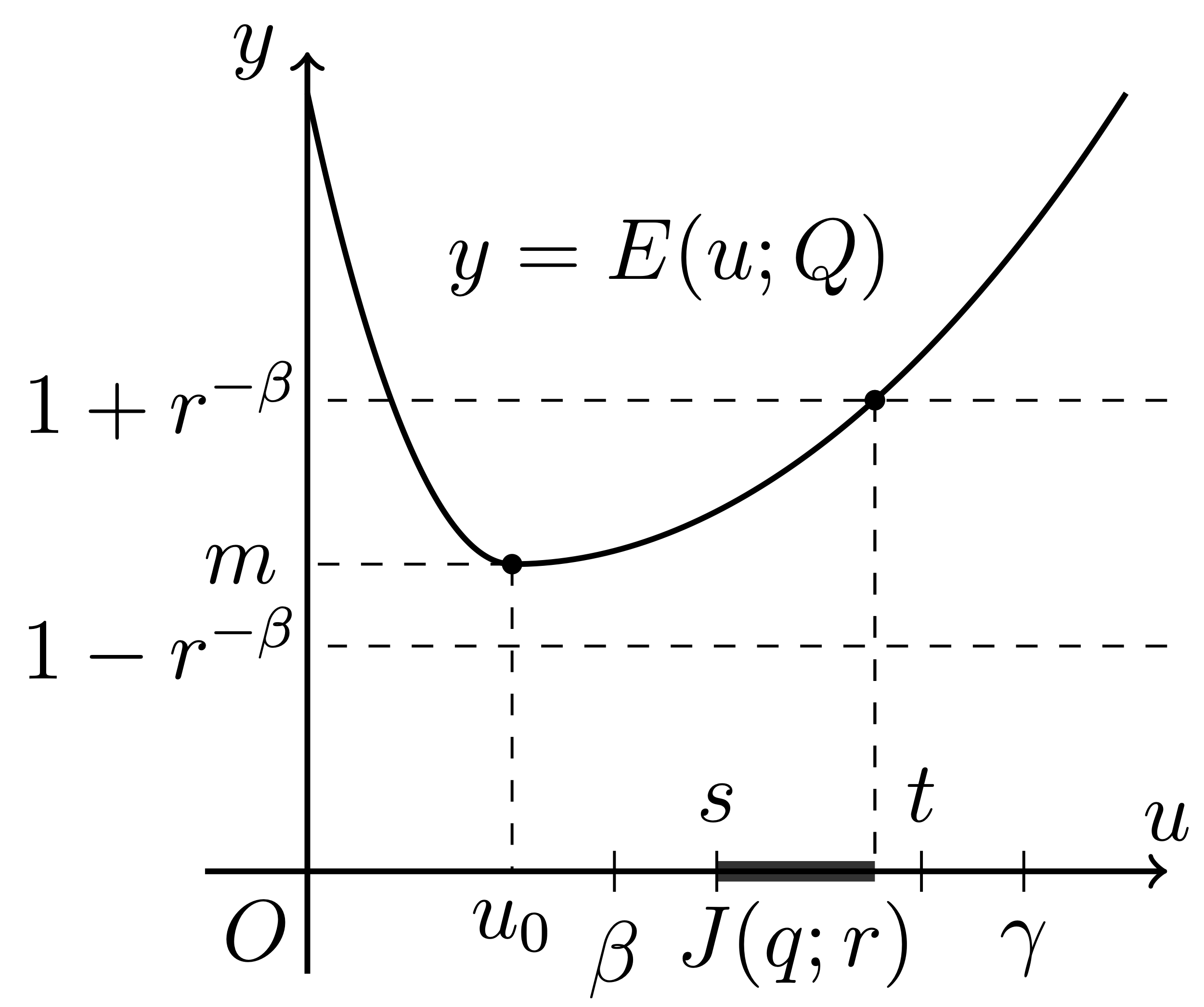}
    \caption{Case~2.0}
    \label{Figure-Case20}
\end{minipage}
\begin{minipage}[b]{0.32\columnwidth}
    \centering
    \includegraphics[width=0.9\columnwidth]{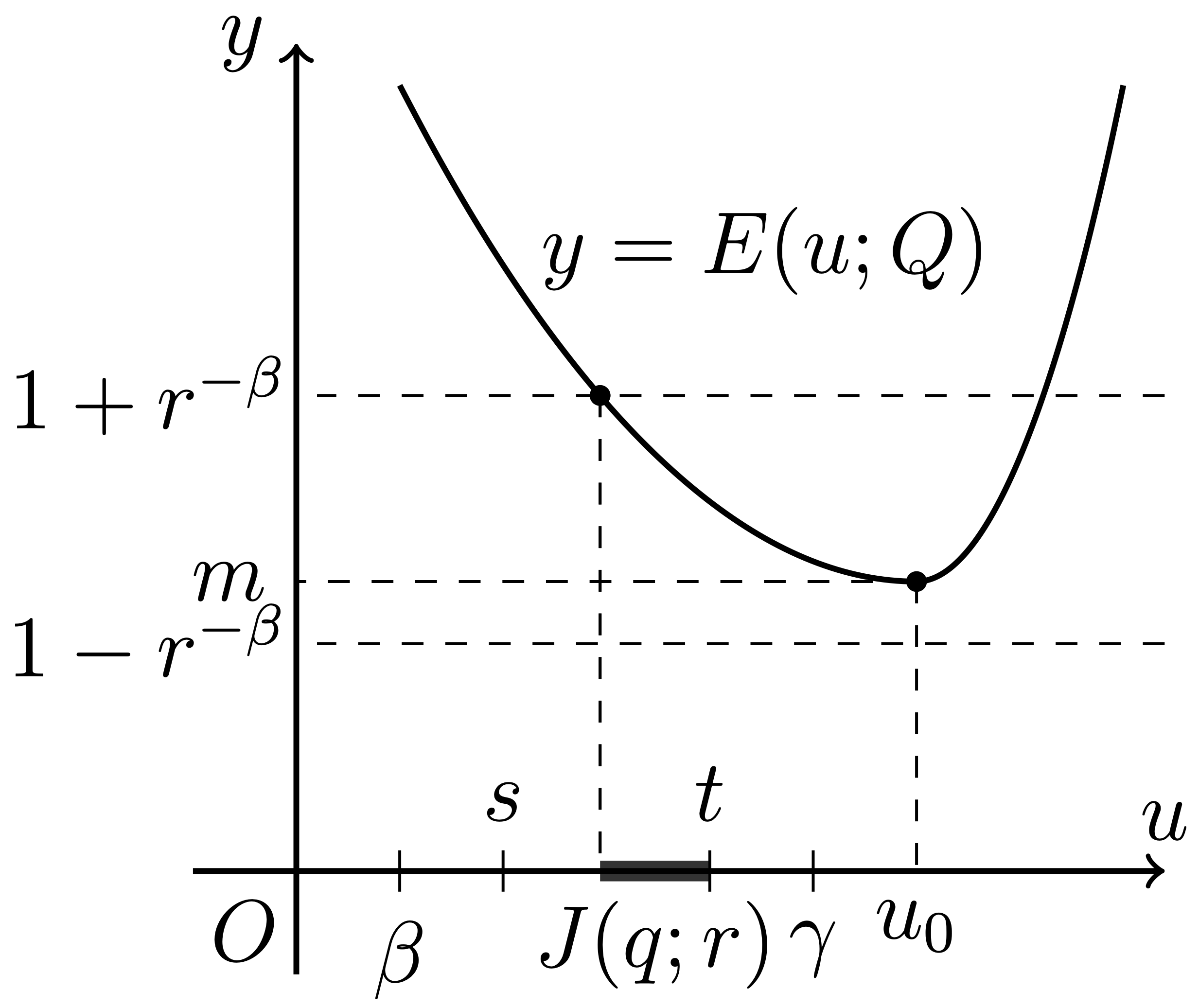}
    \caption{Case~2.1}
    \label{Figure-Case21}
\end{minipage}
\begin{minipage}[b]{0.32\columnwidth}
    \centering
    \includegraphics[width=0.9\columnwidth]{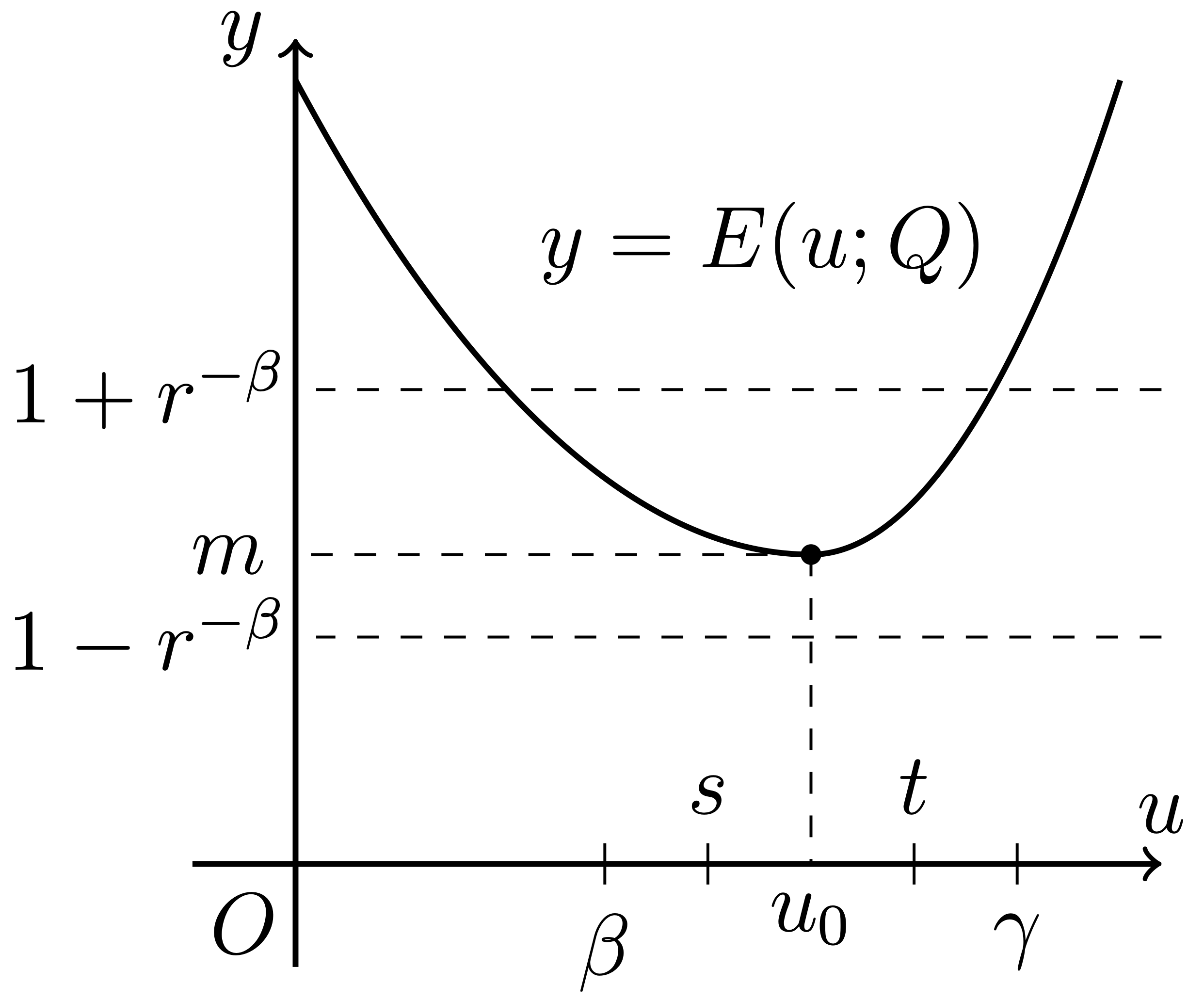}
    \caption{Case~2.2}
    \label{Figure-Case22}
\end{minipage}
\end{figure}
We decompose $C^\nu(M)$ into $\calC^\nu(M)=\calC_{20}^\nu(M)\cup \calC_{21}^\nu(M)\cup \calC_{22}^\nu(M)$, where 
\begin{align*}
\calC_{20}^\nu(M)&=\calC^\nu (M;u_0(q/r)\leq \beta),\\
\calC_{21}^\nu(M)&=\calC^\nu (M;u_0(q/r)\geq \gamma),\\  
\calC_{22}^\nu(M)&=\calC^\nu (M; u_0(q/r)\in (\beta,\gamma)). 
\end{align*}
\begin{lemma}[Case~2.0]\label{Lemma-series31}Let $\sigma> \max ((k+1)/\beta,\ k/(\beta-2) )$.   We have
\begin{align*}
\sum_{J\in \calC_{20}^\nu (M)} (\diam J)^\sigma <\infty .
\end{align*}
\end{lemma}
\begin{proof} Let $q=(q_1,\ldots, q_k)\in I^\nu$ and we suppose that $u_0=u_0 (q/r)\leq \beta$. Then, the function $E(u)=E(u;q/r)$ is monotonically increasing and continuously differentiable on $[s,t]$. Therefore, the inverse function of $E|_{[s,t]}$ 
\[
L\colon [E(s), E(t)]\ni E(u) \longmapsto u\in [s,t] 
\]
exists (See Figure~\ref{Figure-Case20}). The function $L(u)$ is also monotonically increasing and continuously differentiable on $[E(s), E(t)]$. Then by setting $A=[E(s), E(t)]\cap [1-r^{-\beta},1+r^{-\beta}]$, we obtain $J(q; r)=L (A)$. By the mean value theorem, there exists $\theta\in A$ such that
\[
\diam L(A)\leq 2r^{-\beta} L'(\theta) =2r^{-\beta}\frac{1}{E'(L(\theta)) }.
\]
By combining $E'(\beta)\ge 0$, the mean value theorem, and Lemma~\ref{Lemma-SecondDerivE}, there exists $\theta'\in (\beta,L(\theta))$ such that
\[
E'(L(\theta))\geq E'(L(\theta))-E'(\beta)=(L(\theta)-\beta) E'' (\theta')\gg (\log (r/q_k))^2,
\]
where we apply $L(\theta)\geq s$ and $s>\beta$ at the last inequality. Therefore, we obtain $\diam J(q; r)\ll r^{-\beta}(\log (r/q_k))^{-2}$, which implies that
\[
\sum_{J\in \calC_{20}^\nu (M)} (\diam J)^\sigma \ll_\sigma  \sum_{r\geq M} \sum_{q\in I^\nu} r^{-\sigma\beta}\log^{-2\sigma} (r/q_k).
\]
Lemma~\ref{Lemma-logevalution} yields that the right-hand side is 
\[
\ll_\sigma \sum_{r\geq M} r^{-\sigma\beta+k-1} (r +r^{2\sigma} )\log r= \sum_{r\geq M}( r^{-\sigma\beta +k} + r^{-\sigma(\beta-2) +k-1} )\log r
\]
which converges if $-\sigma\beta+k<-1$ and $-\sigma(\beta-2) +k-1<-1$, that is,
\[
\sigma >\max \left( (k+1)/\beta,\  k/(\beta-2)\right). \qedhere
\]
\end{proof}

\begin{lemma}[Case~2.1]\label{Lemma-series32}
 Let $\sigma> \max ((k+1)/\beta,\ k/(\beta-2) )$. We have
\begin{align*}
\sum_{J\in \calC_{21}^\nu (M)} (\diam J)^\sigma <\infty .
\end{align*}
\end{lemma}
\begin{proof}We obtain this lemma in a similar manner to the proof of Lemma~\ref{Lemma-series31}. 
\end{proof}

For Case~2.2, we prepare the following lemma.  
\begin{lemma}\label{Lemma-atmostone}Let $u_0(\cdot)$ and $m(\cdot)$ be as in Section~\ref{Section-Case3smallS}. For all $q=(q_1,\ldots, q_{k-1})\in I^{(\nu_1,\ldots, \nu_{k-1})}$, and $p,p'\in  I_{k}^{(\nu_k)}$ with $p'<p$, if we have   
\begin{align} \label{Elements-atmostone}
u_0(Q_1,\ldots, Q_{k-1}, p/r)&\in (\beta,\gamma),\\ \label{Elements-atmostone2}
u_0(Q_1,\ldots, Q_{k-1}, p'/r)&\in (\beta,\gamma),
\end{align}
then 
\begin{equation}\label{Inequality-m}
m(Q_1,\ldots,Q_{k-1},p'/r)-m(Q_1,\ldots ,Q_{k-1},p/r)\geq b_k/r. 
\end{equation}
In particular, for all $q=(q_1,\ldots, q_{k-1})\in I^{(\nu_1,\ldots, \nu_{k-1})}$ and intervals $H\subseteq \mathbb{R}$,  if 
\[
\diam H< b_k/r,
\]
then at most one $p\in I_{k}^{(\nu_k)}$ with \eqref{Elements-atmostone} satisfies 
$
m(Q_1,\ldots,Q_{k-1}, p/r)\in H
$.
\end{lemma}

\begin{proof}Take arbitrary $p,p'\in I_k^{(\nu_k)}$ with $p'<p$. Suppose that \eqref{Elements-atmostone} and \eqref{Elements-atmostone2} are true. Let $u_0$ and $u_0'$ be the left-hand side of \eqref{Elements-atmostone} and \eqref{Elements-atmostone2}, respectively. Since $E(u_0'; Q_1,\ldots, Q_{k-1},p'/r)$ is the least value of $E(\:\cdot\:; Q_1,\ldots, Q_{k-1},p'/r)$, we see that
\begin{align*}
&m(Q_1,\ldots,Q_{k-1},p/r)-m(Q_1,\ldots ,Q_{k-1},p'/r) \\
&= E(u_0;Q_1,\ldots, Q_{k-1}, p/r ) - E(u_0' ;Q_1,\ldots, Q_{k-1}, p'/r)\\
&\geq E(u_0 ;Q_1,\ldots, Q_{k-1}, p/r )- E(u_0 ;Q_1,\ldots, Q_{k-1}, p'/r)\\
&=b_k(p/r)^{u_0}-b_k  (p'/r)^{u_0}\\
& = u_0 b_k\theta^{u_0-1} (p-p')/r 
\end{align*}
for some $\theta\in (p'/r,p/r)$. Recall that $\nu_k=1$ and $I_k^{(1)}=(r,B_k r)_\mathbb{Z}$, and hence $\theta>p'/r >1$. 
Thus, by $u_0>\beta>1$, we have  
\[
u_0 b_k  \theta^{u_0-1}(p-p')/r \geq b_k/r 
\]
which leads to \eqref{Inequality-m}. 

In particular, by \eqref{Inequality-m}, it is clear that if $\diam H< b_k/r$, then at most one $p\in I_{k}^{(\nu_k)}$ with \eqref{Elements-atmostone} satisfies 
$
m(Q_1,\ldots,Q_{k-1}, p/r)\in H
$.
\end{proof}

Suppose that $u_0\in (\beta,\gamma)$. The function $E(u; Q)$ is decreasing and continuously differentiable on $(-\infty,u_0]$, and hence the inverse function of $E|_{(-\infty,u_0]} $
\[
L_0(\ \cdot\ )=L_0(\ \cdot\ ; Q):[m, \infty)\ni E(u) \longmapsto u\in (-\infty, u_0]
\]
exists (See Figure~\ref{Figure-Case22}). Moreover, the function $E(u; Q)$ is increasing and continuously differentiable on $[u_0,\infty)$. Thus, the inverse function of $E|_{[u_0,\infty)}$
\[
L_1(\ \cdot\ )=L_1(\ \cdot\ ; Q):[m, \infty) \ni E(u) \longmapsto u\in [ u_0, \infty)
\]
exists. The functions $L_0$ and $L_1$ are monotonic and continuously differentiable. Suppose that $J=J(q; r)\in \mathcal{C}_{22}(M)$ and $X$ is a real parameter in $ (2/r^{\beta}, b_k/(2r))$ depending only on $r$, where we will choose $X(r)=r^{2(1-\beta)/3}$ later. Let us consider the following cases:
\begin{itemize}
\item (Case~2.2.0)  $m\in [1-X,1+r^{-\beta} ]$; 
\item (Case~2.2.1)  $m>1+r^{-\beta} $; 
\item (Case~2.2.2)  $m<1-X$.
\end{itemize}
In Figures~\ref{Figure-Case220}, \ref{Figure-Case221}, and \ref{Figure-Case222}, we plot the graph of $y=E(u;Q)$ in each case.

\begin{figure}[htbp]
\centering
\begin{minipage}[b]{0.32\columnwidth}
\centering
\includegraphics[width=0.9\columnwidth]{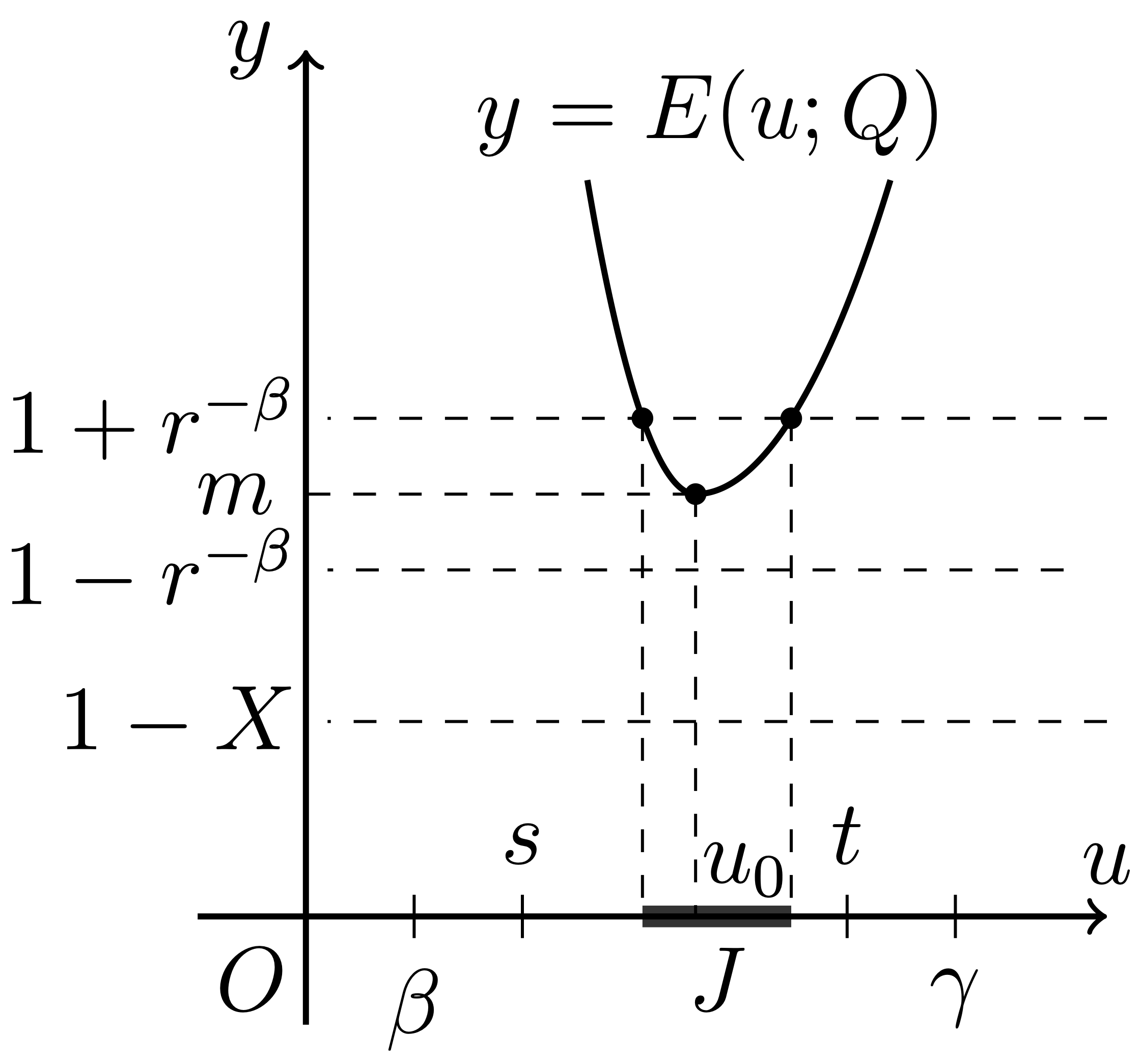}
\caption{Case~2.2.0}
\label{Figure-Case220}
\end{minipage}
\begin{minipage}[b]{0.32\columnwidth}
\centering
\includegraphics[width=0.9\columnwidth]{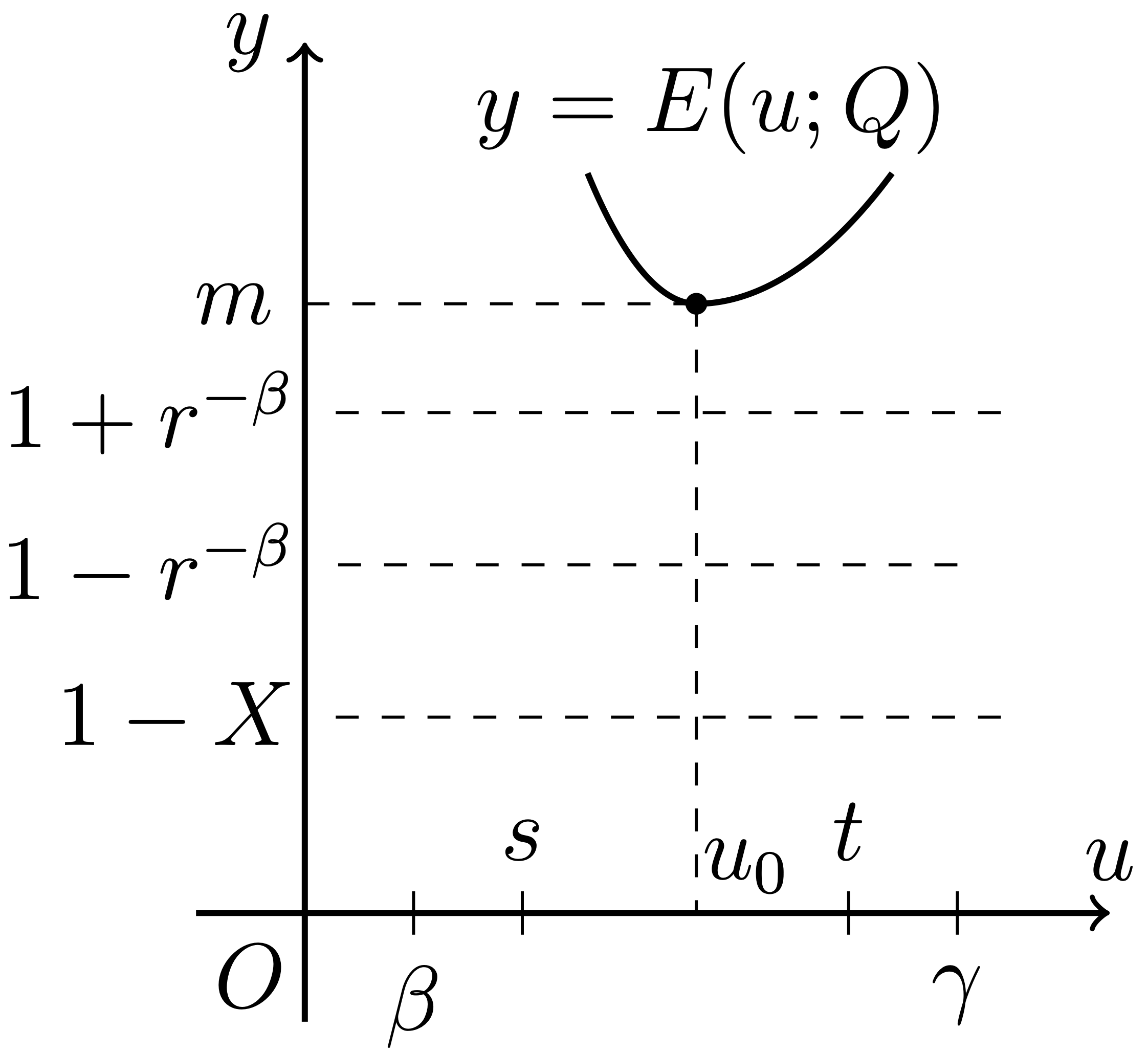}
\caption{Case~2.2.1}
\label{Figure-Case221}
\end{minipage}
\begin{minipage}[b]{0.32\columnwidth}
\centering
\includegraphics[width=0.9\columnwidth]{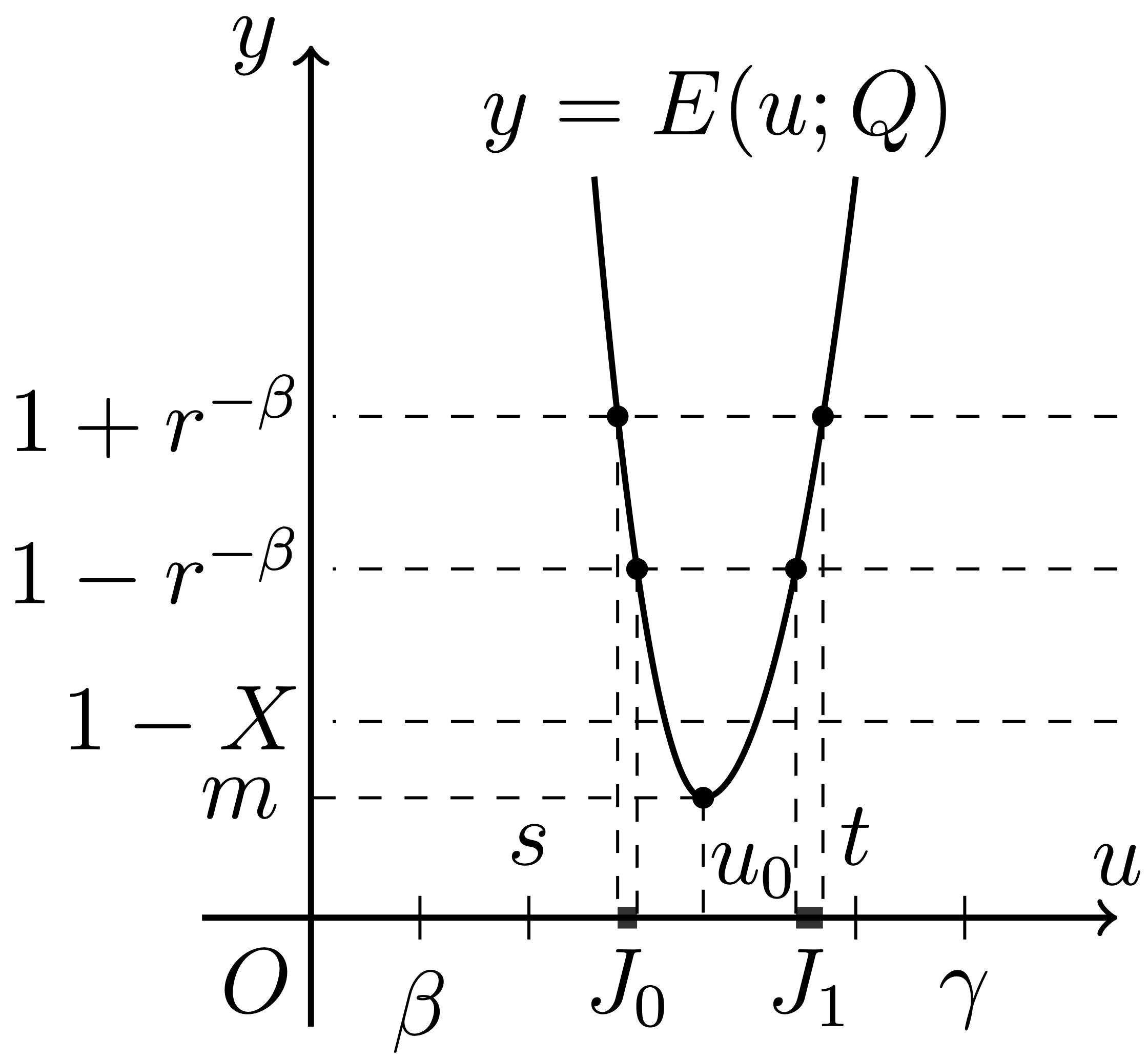}
\caption{Case~2.2.2}
\label{Figure-Case222}
\end{minipage}
\end{figure}

We decompose $\calC_{22}^\nu(M)=\calC_{220}^\nu(M)\cup \calC_{221}^\nu(M)\cup \calC_{222}^\nu(M)$, where 
\begin{align*}
\calC_{220}^\nu (M)&=\calC_{22}^\nu(M; m\in [1-X,1+r^{-\beta}]),\\
\calC_{221}^\nu (M)&=\calC_{22}^\nu(M; m>1+r^{-\beta}),\\
\calC_{222}^\nu (M)&=\calC_{22}^\nu(M; m<1-X).
\end{align*}

\begin{lemma}[Case 2.2.0] \label{Lemma-Case220}
For all $J=J(q;r)\in \calC_{220}^\nu(M)$, we have 
\[
\diam J \ll rX^{1/2},
\]
where the implicit constant is absolute.
\end{lemma}
\begin{proof} Suppose that $J=J(q;r)\in \calC_{220}^\nu(M)$. Then it follows that 
\begin{align} \nonumber
\diam J&=\diam \biggl(E^{-1} ([1-r^{-\beta}, 1+r^{-\beta}])\cap [s,t]\biggr) \\ \nonumber
&\leq \diam \biggl(E^{-1} ([0, 1+r^{-\beta}])\biggr)\\ \nonumber
&\leq L_1(1+r^{-\beta}) -    L_0(1+r^{-\beta} )\\ \label{Inequality-L1L2}
&\leq |L_1(1+r^{-\beta}) - u_0| + |L_0(1+r^{-\beta})-u_0|
\end{align}
Let $\eta_i =L_i(1+r^{-\beta})>0$ for every $i\in \{1,2\}$.  Then by the Taylor expansion and Lemma~\ref{Lemma-SecondDerivE}, there exists a real number $\theta$ in between $\eta_i$ and $u_0$ such that
\begin{align*}
E(\eta_i) &=E(u_0) + (\eta_i -u_0)E'(u_0 ) +  (\eta_i -u_0)^2  E''(\theta ) \\
&=E(u_0)+ (\eta_i -u_0)^2  E''(\theta)\geq E(u_0)+K(\eta_i -u_0)^2 (\log (q_k/r))^2
\end{align*}
for some constant $K>0$. Further, by the condition of Case 2.2.0, we observe that 
\[
E(\eta_i) - E(u_0) = |1+r^{-\beta} -m|\leq X+r^{-\beta},  
\]
which implies that
\begin{equation}\label{Inequality-etau0}
|\eta_i -u_0|^2 \ll \frac{X+r^{-\beta}}{(\log (q_k/r))^2 } \ \ll X\log^{-2} (1+1/r)\ll r^2X,
\end{equation}
where $X\geq 2r^{-\beta}$ and $q_k>r$ lead to the second inequality. Therefore, by combining \eqref{Inequality-L1L2} and \eqref{Inequality-etau0}, we conclude that $\diam J \ll rX^{1/2}$.
\end{proof}

\begin{lemma}[Case 2.2.1]\label{Lemma-Case332}
For all $J\in \calC_{221}^\nu(M)$, we have $\diam J=0$.
\end{lemma}

\begin{proof}
Suppose that $J=J(q;r)\in \calC_{221}^\nu(M)$. Then, it is clear that $J$ is empty since $E(u;q/r)\geq m(q/r)>1+r^{-\beta}$ for all $u\in \mathbb{R}$ and $J(q;r)\subseteq E^{-1}([1-r^{-\beta}, 1+r^\beta])$.
\end{proof}

Before discussing Case~2.2.2, for all $J=J(q;r)\in \calC_{222}^\nu (M)$ and $i\in\{0,1\}$, we define 
\begin{equation}\label{definition-J}
J_i =J_i(q;r)= L_i\biggl([1-r^{-\beta}, 1+r^{-\beta}]\cap[m, \infty ) ; q/r\biggr)\cap [s,t]. 
\end{equation}
\begin{lemma}[Case~2.2.2]\label{Lemma-Case333}
For all $J=J(q;r)\in \calC_{222}^\nu (M)$ and $i\in \{0,1\}$,
\begin{itemize}
\item if $m\in [1-b_k/(2r), 1-X]$, then 
\begin{equation}\label{Inequality-case333main1}
\diam J_i\ll  r^{-\beta}(\log r) X^{-1} \log^{-2} (q_k/r),
\end{equation}
where the implicit constant is absolute\textup{;}
\item if $m= [1-(\ell+1)b_k/(2r), 1- \ell b_k/(2r)]$ for some $1\leq \ell\leq 2r/b_k$, then 
\begin{equation}\label{Inequality-case333main2}
\diam J_i\ll  r^{-\beta+1}(\log r) \ell^{-1} \log^{-2} (q_k/r),
\end{equation}
where the implicit constant is absolute.
\end{itemize}

\end{lemma} 
\begin{proof} Suppose that $J=J(q;r)\in \calC_{222}(M)$. Then, we have $m=m(q/r)<1-X$ and
\begin{align*}
J=J(q;r)= E^{-1} ([1-r^{-\beta}, 1+r^{-\beta}]; q/r)\cap [s,t]=J_0 \cup J_1. 
\end{align*}
Take any $i\in\{0,1\}$. Then by the mean value theorem, there exists $\theta\in [1-r^{-\beta}, 1+r^{-\beta}]$ with $L_i( \theta)\in [s,t]$ such that
\begin{equation}\label{Inequality-0case333}
\diam J_i \leq 2r^{-\beta} |L'_i (\theta )|= \frac{2r^{-\beta}}{ |E'(L_i(\theta))|}.
\end{equation}
Applying the mean value theorem again, there exists $\theta'$ in between $u_0$ and $L_i(\theta)$ such that
\[
E'(L_i(\theta) )=E'(u_0) + (L_i(\theta)-u_0 ) E''(\theta' )=(L_i(\theta)-u_0 ) E''(\theta' ),
\]
and hence Lemma~\ref{Lemma-SecondDerivE} implies that
\begin{equation}\label{Inequality-1case333}
|E'(L_i(\theta) )| \geq |L_i(\theta)-u_0| (\log (q_k/r))^2.
\end{equation}
Further, there exists $\theta''$ in between $u_0$ and $L_i (\theta)$ such that
\begin{equation*}
|\theta -m |  = |E (L_i(\theta) )-E(u_0)  |= |L_i(\theta)-u_0| |E'(\theta'')|.
\end{equation*}
   We note that $\beta \leq \theta''\leq \gamma$ since
\[
\beta \leq \min (u_0,  L_i (\theta)) \leq \theta'' \leq \max (u_0,  L_i (\theta))\leq  \gamma.
\]
If $E'(\theta'')\geq 0$, then since $Q_i\ll 1$ for every $i\in [k]$, we have 
\begin{equation}\label{Inequality-4case333}
|E'(\theta'')|=E'(\theta'')= \sum_{ i\in [k]} b_iQ_i^{\theta''} (\log Q_i) \leq  \sum_{\substack {i\in [k] \\ \text{ with } Q_i>1}} b_iQ_i^{\gamma} (\log Q_i)  \ll 1.
\end{equation}
If $E'(\theta'')\leq 0$, then similarly, 
\begin{equation}\label{Inequality-5case333}
|E'(\theta'')|= -E'(\theta'')=\sum_{ i\in[k]} b_iQ_i^{\theta''} \log (1/Q_i)  \leq \sum_{\substack {i\in [k] \\ \text{ with } Q_i<1}} b_iQ_i^{\beta} (\log (1/Q_i))  \ll \log r.
\end{equation}
Therefore, by combining the above, we have
\begin{align} \label{Inequality-6case333}
\diam J_i
&\ll \frac{r^{-\beta}}{ |E'(L_i(\theta))|} \ll \frac{r^{-\beta}}{ |L_i(\theta)-u_0| (\log (q_k/r))^2}\\ \nonumber
&\ll r^{-\beta} |E'(\theta'')||\theta-m|^{-1} \log^{-2} (q_k/r) \\ \nonumber
&\ll r^{-\beta}(\log r) |\theta-m|^{-1} \log^{-2} (q_k/r).
\end{align}

If $m\in [1-b_k/(2r), 1-X]$, then by $\theta\in [1-r^{-\beta},1+r^{-\beta}]$ and  $X\geq 2r^{-\beta}$, we have
\[
|\theta-m|\geq X -r^{-\beta} \gg X.
\] 
This inequality and \eqref{Inequality-6case333} lead to \eqref{Inequality-case333main1}. 

If $m\in [1-(\ell+1)b_k/(2r), 1- \ell b_k/(2r)]$ for some $1\leq \ell\leq 2r/b_k$, then we have $\ell/r \ll  |\theta -m |$, where the implicit constant does not depend on $\ell$. Therefore, by \eqref{Inequality-6case333}, we obtain $\diam J_i\ll  r^{-\beta+1}(\log r)\: \ell^{-1} \log^{-2} (q_k/r)$, which is \eqref{Inequality-case333main2}. 
\end{proof}

\begin{lemma}\label{Lemma-C331333} Let $\sigma\in (0,1)$. We have
\begin{align} \label{Inequality-C331}
&\sum_{J\in \mathcal{C}_{220}^\nu (M) }(\diam J)^\sigma \ll_\sigma  \sum_{r\geq M} r^{\sigma+k-1} X^{\sigma/2},  \\ \label{Inequality-C333} 
&\sum_{J\in \mathcal{C}_{222}^\nu (M)}  \sum_{i\in \{0,1\}}(\diam J_i)^\sigma \ll_\sigma \sum_{r\geq M}\left( r^{-\sigma\beta}+r^{-\sigma(\beta-3)-1}+\frac{r^{-(\beta-2)\sigma-1}}{X^{\sigma}}\right)r^k\log^{\sigma+1} r. 
\end{align}
\end{lemma}
\begin{proof}Take arbitrary $\sigma\in (0,1)$. Since the length of $[1-X,1+r^{-\beta}]$ is less than $b_k/r$ by $X<b_k/(2r)$,  Lemmas~\ref{Lemma-atmostone} and \ref{Lemma-Case220} imply that  
\begin{align*}
\sum_{J\in \mathcal{C}_{220}^\nu (M) }(\diam J)^\sigma &= \sum_{r\geq M} \sum_{q_1,\ldots, q_{k-1} } \sum_{\substack{q_k \in I^{(\nu_k)}_k \text{with}\\ u_0\in (\beta,\gamma), \\ m\in [1-X,1+r^{-\beta} ] } }  (\diam J(q;r))^\sigma\\
&\ll_\sigma \sum_{r\geq M} r^{\sigma+ k-1} X^{\sigma/2},
\end{align*}
where $(q_1,\ldots, q_{k-1})$ runs over $I_1^{(\nu_1)}\times \cdots \times I_{k-1}^{(\nu_{k-1})}$. Therefore, we conclude 
\eqref{Inequality-C331}. 

Let us next prove \eqref{Inequality-C333}. We see that
\begin{align*}
&\sum_{J\in \mathcal{C}_{222}^\nu (M)}  \sum_{i\in \{0,1\}}(\diam J_i)^\sigma\\
&\leq \sum_{i\in \{0,1\}}\sum_{r\geq M} \sum_{q_1,\ldots ,q_{k-1}}\sum_{\substack{q_k \in I^{(\nu_k)}_k\text{with}\\ u_0\in (\beta,\gamma) \\ m\in [1-b_k/(2r),1-X ] } } (\diam J_i)^\sigma\\
&+\sum_{i\in \{0,1\}}\sum_{r\geq M} \sum_{q_1,\ldots ,q_{k-1}}\sum_{1\leq \ell\leq 2r/b_k} \sum_{\substack{q_k \in I^{(\nu_k)}_k \text{with}\\ u_0\in (\beta,\gamma) \\ m\in [1-(\ell+1)b_k/(2r), 1- \ell b_k/(2r)] } } (\diam J_i)^\sigma\\
&\eqqcolon S_0+S_1.
\end{align*}
By \eqref{Inequality-case333main1} and Lemma~\ref{Lemma-atmostone}, we obtain 
\begin{align*}
S_0&\ll\sum_{r\geq M} r^{-\beta\sigma+k-1}(\log^\sigma r) X^{-\sigma} \log^{-2\sigma} ((1+r)/r)\ll \sum_{r\geq M} r^{-(\beta-2)\sigma+k-1}(\log^\sigma r) X^{-\sigma}. 
\end{align*}
To estimate the upper bounds for $S_1$, we define
\begin{gather*}
H_\ell=[1-(\ell+1)b_k/(2r), 1- \ell b_k/(2r)],\\
\begin{aligned}
&A(q_1,\ldots, q_{k-1};r)=\{\ell\in[1,2r/b_k]_{\mathbb{Z}}\colon \text{ there exists }q_k\in I^{(\nu_k)}_k\text{ such that}\\ 
&m(Q_1,\ldots, Q_{k-1}, q_k/r)\in H_\ell \quad \& \quad u_0(Q_1,\ldots, Q_{k-1}, q_k/r)\in (\beta,\gamma) \}.
\end{aligned}
\end{gather*}

Let $A(q_1,\ldots, q_{k-1};r)=\{\ell_1,\ldots, \ell_h\}$ with $\ell_1<\cdots<\ell_h$ and $h\leq 2r/b_k$. By combining Lemma~\ref{Lemma-atmostone} and $\diam H_{\ell}=b_k/(2r)$, for every $j \in [h]$ there uniquely exists $q_k \in I_k^{(\nu_k)}$ such that 
\begin{equation}\label{eq-optional1}
m(Q_1,\ldots, Q_{k-1}, q_k/r)\in H_{\ell_j}\quad \& \quad u_0(Q_1,\ldots, Q_{k-1}, q_k/r)\in (\beta,\gamma). 
\end{equation}
Let $q_k(j)$ be such $q_k$. Then, Lemma~\ref{Lemma-atmostone} implies that $q_k(j')>q_k(j)$ for all  $1\leq j'<j\leq k$ since by \eqref{eq-optional1}, we have
\[
m(Q_1,\ldots, Q_{k-1}, q_k(j')/r) - m(Q_1,\ldots, Q_{k-1}, q_k(j)/r)\geq 0.
\]
Therefore, by \eqref{Inequality-case333main2},  we obtain
\begin{align*}
S_1&= \sum_{i\in \{0,1\}}\sum_{r\geq M} \sum_{q_1,\ldots ,q_{k-1}}\sum_{\ell\in A(q_1,\ldots,q_{k-1};r)}\sum_{\substack{q_k \in I^{(\nu_k)}_k\: \text{with} \\u_0\in (\beta,\gamma) \\ m\in H_\ell } } (\diam J_i)^\sigma\\
&\ll_\sigma \sum_{r\geq M} \sum_{q_1,\ldots, q_{k-1}} \sum_{j=1}^h   r^{-\sigma(\beta-1)}(\log^\sigma r) \ell_j^{-\sigma} \log^{-2\sigma} (q_k(j)/r).
\end{align*}
Here, we have $q_k(j)\geq r+j$ for every $j\in [h]$ since $q_k(j)\in I_k^{(1)}=(r, B_kr)$ and $(q_k(j))_{j=1}^h$ is monotonically increasing.  Furthermore, we also have $\ell_j\geq j$ since $(\ell_j)_{j=1}^h$ is monotonically increasing. Therefore, by $h\leq 2r/b_k$, we obtain
\begin{align*}
S_1&\ll \sum_{r\geq M} \sum_{q_1,\ldots, q_{k-1}} \sum_{j=1}^h r^{-\sigma(\beta-1)}(\log^\sigma r) j^{-\sigma} \log^{-2\sigma} (1+j/r) \\
&\ll_\sigma \sum_{r\geq M} r^{k-1} r^{-\sigma(\beta-3)}(\log^\sigma  r)\sum_{1\leq j\leq 2r/b_k} j^{-3\sigma}\\
&\ll_\sigma \sum_{r\geq M} r^{-\sigma(\beta-3)+k-1}(\log^\sigma r) (1+r^{1-3\sigma}\log r)\\
&\ll \sum_{r\geq M} (r^{-\sigma(\beta-3)+k-1}(\log^\sigma r) + r^{-\sigma\beta+k}(\log^{\sigma+1} r) ).
\end{align*}
By combining $S_0$ and $S_1$, we conclude \eqref{Inequality-C333}.
\end{proof}

\section{Proof of Theorem~\ref{Theorem-genMain2}}\label{Section-Proof}
Fix any $n\in \mathbb{N}$. Let $\beta$, $s$, $t$, $\gamma$ be real numbers with $3n+4<\beta <s<t<\gamma$. Take arbitrary $k\in [n]$ and positive real numbers $b_1,\ldots, b_k$. By Lemma~\ref{Lemma-AtoB1}, it suffices to show that 
\begin{equation}\label{Equation-Goal}
\Haus \mathcal{B}_{b_1,\ldots, b_k}(s,t) \leq  3k/(s-4)
\end{equation}
Let $\sigma\in (3k/(\beta-4),1)$. Let $M_0$ be a sufficiently large parameter, and take $M\geq M_0$.  Let $\calC^\nu (M)$ be as in Section~\ref{Section-General} for every $\nu \in \{0,1\}^k$. We recall that 
\begin{equation}\label{Equation-Case0}
\mathcal{B}(s,t)\subseteq  \bigcup_{\nu \in \{0,1\}^k} \calC^\nu (M)=\calC_0 (M) \cup  \calC_1 (M) \cup \calC_2 (M)
\end{equation}
to obtain  
\[
\sum_{J \in \bigcup_{\nu \in \{0,1\}^k} \calC^\nu (M)} (\diam J)^\sigma   \leq \left(\sum_{J\in \calC_0 (M) \cup  \calC_1 (M)} + \sum_{J\in \calC_2 (M)}\right) (\diam J)^\sigma. 
\] 
Since $\sigma>3k/(\beta-4)>(k+1)/\beta$,  the inequality \eqref{Series-Case1Case2} implies 
\begin{equation}
\sum_{J\in \calC_0 (M) \cup \calC_1 (M) } (\diam J)^\sigma<\infty.
\end{equation}
Moreover, by recalling the discussion in Section~\ref{Section-nonsimple}, we decompose 
\begin{equation}\label{Equation-Case3}
\calC_2 (M)=\bigcup_{\nu} (\calC_{20}^\nu(M)\cup \calC_{21}^\nu(M)\cup \calC_{22}^\nu(M)), 
\end{equation}
 where $\nu=(\nu_1,\ldots, \nu_k)$ runs over $\{0,1\}^k$ with $\#\{\nu_1,\ldots, \nu_k\}>1$.  Take an arbitrary such $\nu$. Without loss of generality, we may assume that $\nu_k=1$ by renumbering. 

 Since $\sigma>3k/(\beta-4)>\max ((k+1)/\beta,\ k/(\beta-2) )$, Lemmas~ \ref{Lemma-series31} and \ref{Lemma-series32} lead to
\begin{equation*}
\sum_{J\in \calC_{20}^\nu (M) \cup \calC_{21}^\nu (M) } (\diam J)^\sigma <\infty.
\end{equation*}

Let us next discuss $\calC_{22}^\nu(M)$.  Let $X$ be a real parameter in $(2r^{-\beta}, b_k/(2r))$ depending only on $r$, where we will choose $X=r^{2/3(1-\beta)}$.  Let $J_0$ and $J_1$ be as in \eqref{definition-J}. By Lemmas~\ref{Lemma-Case332} and \ref{Lemma-C331333}, we have
\begin{align}\label{Inequality-upperCase33}
&\sum_{J\in \calC_{22}^\nu (M)} (\diam J)^{-\sigma} \\ \nonumber
&=\sum_{J\in \calC_{220}^\nu(M)}  (\diam J)^\sigma + \sum_{J\in \calC_{221}^\nu(M)}  (\diam J)^\sigma+ \sum_{J\in \calC_{222}^\nu(M)} \sum_{i\in \{0,1\}} (\diam J_i)^\sigma\\ \nonumber
&\ll \sum_{r\geq M} \biggl(r^{\sigma+k-1} X^{\sigma/2}  +( r^{-\sigma\beta}+r^{-\sigma(\beta-3)-1}+r^{-(\beta-2)\sigma-1}X^{-\sigma})r^k\log^{\sigma+1} r\biggr). 
\end{align}
To be optimized, we choose $X=r^{2(1-\beta)/3}$ by solving $r^{\sigma+k-1} X^{\sigma/2}=r^{-(\beta-2)\sigma+k-1}X^{-\sigma}$. Note that $X\in (2r^{-\beta},b_k/(2r))$ holds since $M_0$ is sufficiently large and $\beta>3n+4\geq 7$. Therefore, the most right-hand side of \eqref{Inequality-upperCase33} is 
\[
\ll \sum_{r\geq M}  (r^{-\sigma\beta}+r^{-\sigma(\beta-3)-1}+ r^{-(\beta-4)\sigma/3-1} )r^k\log^{\sigma+1} r. 
\]
By $\beta>5/2$, we have $-\sigma(\beta-3)-1< -(\beta-4)\sigma/3-1$.  Further,  $\sigma>3k/(\beta-4)$ implies that $-\sigma\beta\leq   -(\beta-4)\sigma/3-1$.  Thus, we obtain
\begin{equation}\label{Inequality-allC3}
\sum_{J\in \calC_{22}^\nu (M)} (\diam J)^{-\sigma} \ll  \sum_{r\geq M} r^{-(\beta-4)\sigma/3+k-1} \log^{\sigma+1} r
\end{equation}
which converges if $-(\beta-4)\sigma/3+k-1<-1$ \textit{i.e.} $\sigma>3k/(\beta-4)$.

By combining \eqref{Equation-Case0} to \eqref{Inequality-allC3} and by taking $M\to \infty$, we have $\mathcal{H}_\infty^\sigma (\mathcal{B}(s,t) )=0$ for all $\sigma>3k/(\beta-4) $. Thus, the definition of the Hausdorff dimension implies $\Haus\mathcal{B}(s,t) \leq 3k/(\beta-4)$. By choosing $\beta\to s$, we conclude \eqref{Equation-Goal}. Furthermore, $\Haus \mathcal{B}(s,t)<1$ implies $\mathcal{L}( \mathcal{B}(s,t))=0$ by \eqref{H3}. Therefore, we conclude Theorem~\ref{Theorem-genMain2}.

\section*{Acknowledgement}
The author was supported by JSPS KAKENHI Grant Numbers JP22KJ0375 and JP25K17223.

\end{document}